\documentclass{tac}

\usepackage[all]{xy}
\input diagxy

\usepackage[colorlinks=true]{hyperref}
\hypersetup{allcolors=[rgb]{0.1,0.1,0.4}}

\author{Francesco Dagnino and Fabio Pasquali}

\thanks{acks...} 

\address{
DIBRIS - Universit\`{a} di Genova, Italy \\[5pt]
DIMA - Universit\`{a} di Genova, Italy 
}

\title{Cauchy Completions and the Rule of Unique Choice in Relational Doctrines}

\copyrightyear{2024}


\eaddress{francesco.dagnino@dibris.unige.it \CR pasquali@dima.unige.it}

\newtheorem{theorem}{Theorem} 
\newtheorem{lemma}{Lemma}
\newtheorem{proposition}{Proposition}
\newtheorem{corollary}{Corollary}

\newtheoremrm{remark}{Remark}
\newtheorem{example}{Example} 

\usepackage{etoolbox} 
\usepackage{xcolor} 
\usepackage{xspace}
\usepackage{amsmath}
\usepackage{bbold}
\usepackage{mathtools}

\def\eg{{\em e.g.}\xspace} \def\ie{{\slshape i.e.}\xspace}
\def\cf{cf.~} 
\newcommand{\refItem}[2]{\cref{#1}(\ref{#1:#2})}

\def\ple#1{\ensuremath{\langle #1 \rangle }}

\def\vuoto{}
\newcommand{\FUN}[4]{\ensuremath{{#2}#1{#3}\rightarrow{#4}}}
\newcommand{\fun}[3]{\relax\def\testa{#1}\relax\ifx\testa\vuoto
  \relax\FUN{}{}{{#2}}{{#3}}\else\relax\FUN{:}{{#1}}{{#2}}{{#3}}\fi}

\newcommand{\id}{\mathsf{id}}
\newcommand{\PW}[2][]{\ensuremath{\mathop{\mathscr{P}_{#1}{#2}}}}
\newcommand{\pw}[1]{\relax\def\testa{#1}\relax\ifx\testa\vuoto
  \relax\PW{}\else\relax\PW{\left(#1\right)}\fi}
\newcommand{\fpw}[1]{\relax\def\testa{#1}\relax\ifx\testa\vuoto
  \relax\PW[\omega]{}\else\relax\PW[\omega]{\left(#1\right)}\fi}
\newcommand{\abs}[1]{\left|#1\right|}

\newcommand{\nat}[3]{#1:#2\stackrel.\to#3}

\newcommand{\N}{\mathbb{N}}
\newcommand{\R}{\mathbb{R}} 
\newcommand{\RPos}{\R_{\ge 0}} 
\newcommand{\Bool}{\mathbb{B}}
\newcommand{\PP}{\pw{}} 
 
\newcommand{\Rel}{\mathsf{Rel}}

\newcommand{\scun}{\mathbf{\ni}} 
\newcommand{\sun}{\sigma}

\def\tt{\ensuremath{\textsf{t\kern-.3ex t}}}
\def\ff{\ensuremath{\textsf{f\kern-.3ex f}}}

\let\ForalL\forall \def\Forall#1.{\ForalL_{#1}}
\let\ExistS\exists \def\Exists#1.{\ExistS_{#1}}

\DeclareFontFamily{OT1}{pzc}{}
\DeclareFontShape{OT1}{pzc}{m}{it}{<->s*[1.30]pzcmi7t}{}
\DeclareMathAlphabet{\mathpzc}{OT1}{pzc}{m}{it}
\def\ct#1{\ensuremath{\mathpzc{#1}}}
\def\Ct#1{\ensuremath{\mathbf{#1}}}

\def\op{^{\mbox{\normalfont\scriptsize op}}}
\def\rop{^{\mbox{\normalfont\scriptsize o}}}
\def\co{^{\mbox{\normalfont\scriptsize co}}}
\newcommand{\bop}[1]{(#1\times#1)\op} 
\def\blank{\mathchoice{\mbox{--}}{\mbox{--}}
{\mbox{\scriptsize--}}{\mbox{\tiny--}}}

\newcommand{\CC}{\ct{C}\xspace}

\newcommand{\D}{\ct{D}\xspace}

\newcommand{\oneAr}[3]{\ensuremath{{#1}:{#2}\rightarrow{#3}}}
\newcommand{\twoAr}[3]{\ensuremath{{#1}:{#2}\Rightarrow{#3}}}
\def\tdot{\textbf{.}}
\def\lsta{\vrule depth4pt width0pt}
\def\lstb{\vrule height5pt width0pt}
\def\arnta[#1]{\ar[#1]|-*=0[@]{\lsta\tdot}}
\def\arntb[#1]{\ar[#1]|-*=0[@]{\lstb\tdot}}
\newcommand{\nt}[3]{\ensuremath{{#1}:{#2} \stackrel{\makebox{\kern-.3ex\tdot}}\rightarrow{#3}}}
\newcommand{\lnt}[3]{\ensuremath{{#1}:{#2} \stackrel{\makebox{\kern-.3ex\tdot}}\rightarrow_l{#3}}}
\newcommand{\Hom}[3]{\ensuremath{{#1}({#2},{#3})}}
\newcommand{\Id}{\mathsf{Id}}

\newcommand{\Set}{\ct{Set}\xspace}

\newcommand{\Pos}{\ct{Pos}\xspace}
\newcommand{\Top}{\ct{Top}\xspace} 
\newcommand{\TopDoc}{\textbf{cl}_\beta\xspace}

\newcommand{\RDtn}{\Ct{RD}\xspace} 
\newcommand{\RDtnRuc}{\Ct{RD}_!\xspace}

\newcommand{\RtoRuc}{Ruc\xspace}

\newcommand{\order}{\leq} 
\newcommand{\PDoc}{P}

\newcommand{\RDoc}{R}
\newcommand{\SingAr}{\mathsf{S}}
\newcommand{\Sing}{\fn{\SingAr}}
\newcommand{\SComplAr}{\mathsf{C}}
\newcommand{\SCompl}{\fn{\SComplAr}}
\newcommand{\SDoc}{S} 
\newcommand{\reidx}[1]{_{#1}}

\newcommand{\fn}[1]{\widehat{#1}}
\newcommand{\lift}[1]{\overline{#1}} 
\newcommand{\Ruc}[1]{\mathsf{Map}^{#1}} 
 
\newcommand{\RelCatSet}{\ct{Rel}} 
\newcommand{\bijective}[1]{\ensuremath{#1}-bijective\xspace} 
\newcommand{\injective}[1]{\ensuremath{#1}-injective\xspace} 
\newcommand{\surjective}[1]{\ensuremath{#1}-surjective\xspace}

\newcommand{\relr}{\alpha}
\newcommand{\rels}{\beta}
\newcommand{\relt}{\gamma} 
\newcommand{\eqrelr}{\rho}

\newcommand{\rid}{\mathsf{d}}
\newcommand{\rcomp}{\mathop{\mathbf{;}}}
\newcommand{\rconv}{^{\bot}} 
\newcommand{\rdconv}{^{\bot\bot}} 
\newcommand{\gr}[1]{\Gamma_{#1}} 
 
\newcommand{\GrAr}[1][]{\mathbf{\Gamma}_{#1}} 
\newcommand{\IncAr}[1][]{\mathbf{\iota}_{#1}} 
\newcommand{\exteq}{\approx} 
\newcommand{\RUC}{\textsc{(ruc)}\xspace}
\newcommand{\SRUC}{\textsc{(sruc)}\xspace}

\newcommand{\cl}{\textbf{cl}\xspace}

\newcommand{\complete}[2][]{#2\ifblank{#1}{}{^{#1}}_{!}} 

\newcommand{\VRel}[1]{{#1}\text{-}\mathsf{Rel}}

\newcommand{\Qtl}{V}

\newcommand{\Car}[1]{|#1|}
\newcommand{\qord}{\preceq}
\newcommand{\qmul}{\cdot}
\newcommand{\qone}{1} 
\newcommand{\qsup}{\bigvee}

\newcommand{\dist}{\delta}
\newcommand{\Met}{\ct{Met}}
\newcommand{\Ban}{\ct{Ban}}

\newcommand{\Sub}{\mathsf{Sub}}
\newcommand{\BiMod}{\mathsf{Bimod}}

\newcommand{\SubDoc}{\mathsf{SubRel}}
\newcommand{\MetDoc}{\mathsf{M}}

\newcommand{\BCat}[1]{#1\text{-}\ct{Cat}}
\newcommand{\BCatss}[1]{#1\text{-}\ct{Cat}_{ss}}
\newcommand{\Per}[1]{\ct{Per}_{#1}}

\def\RB#1{\mathchoice
  {\rotatebox[origin=c]{180}{$#1$}}
  {\rotatebox[origin=c]{180}{$#1$}}
  {\rotatebox[origin=c]{180}{$\scriptstyle#1$}}
  {\rotatebox[origin=c]{180}{$\scriptscriptstyle#1$}}}
\def\Ex{\RB{E}\kern-.3ex}
\def\Al{\RB{A}\kern-.6ex}

\newcommand{\Map}[1]{\ct{Map}(#1)}

\newcommand{\SVecDoc}{\mathsf{SV}} 
\newcommand{\vecx}{{\bf x}}
\newcommand{\vecy}{{\bf y}} 
\newcommand{\vecz}{{\bf z}}
\newcommand{\veczero}{{\bf 0}} 
\newcommand{\Norm}[2][]{\|#2\|_{#1}}

\newcommand{\un}{\eta}
\newcommand{\cun}{\epsilon}

\newcommand{\mnd}{\mathbb{T}} 
\newcommand{\mfun}{T}
\newcommand{\mun}{\eta}
\newcommand{\mmul}{\mu} 
\newcommand{\EM}[2]{#1^{#2}} 
\newcommand{\CHEM}[2]{\EM{#1}{#2\mathsf{ch}}} 
\newcommand{\SP}[1]{\ensuremath{#1\text{-}\ct{Sp}}\xspace} 
\newcommand{\CHSP}[1]{\ensuremath{#1\text{-}\ct{Sp}_{\mathsf{ch}}}\xspace} 
\newcommand{\ralg}{\phi}
\newcommand{\ralgb}{\psi}
\newcommand{\SPFun}{G} 
\newcommand{\CHSPFun}{\SPFun_{\mathsf{ch}}} 
\newcommand{\SCFun}{SC} 
\newcommand{\rtc}[2][]{J\ifblank{#1}{}{_{#1}}^\star(#2)}

\usepackage[capitalize]{cleveref}

  \crefname{fig}{Figure}{Figures}
  \crefname{thm}{Theorem}{Theorems}
  \crefname{lem}{Lemma}{Lemmas}
  \crefname{def}{Definition}{Definitions}
  \crefname{prop}{Proposition}{Propositions}
  \crefname{cor}{Corollary}{Corollaries}
  \crefname{ex}{Example}{Examples}
  \crefname{rem}{Remark}{Remarks}
  \crefname{asm}{Assumption}{Assumptions}

\begin{document}

\maketitle
\begin{abstract}
Lawvere's generalised the notion of complete metric space to the field of enriched categories: 
an enriched category is said to be Cauchy-complete if 
every left adjoint bimodule into it is represented by an enriched  functor.  
Looking at this definition from a logical standpoint, 
regarding bimodules as an abstraction of relations and functors as an abstraction of functions, 
Cauchy-completeness resembles a formulation of the rule of unique choice. 
In this paper, we make this analogy precise, using the language of relational doctrines, 
a categorical tool that provides a functorial description of the calculus of relations, in the same way Lawvere's hyperdoctrines give a functorial description of predicate logic. 
Given a relational doctrine, 
we define Cauchy-complete objects as those objects of the  domain category satisfying the rule of unique choice. 
Then, we present a universal construction that completes a relational doctrine with the rule of unique choice, that is, 
producing  a new relational doctrine where all objects are Cauchy-complete. 
We also introduce relational doctrines with singleton objects and show that these have the minimal structure needed to build the reflector of the full subcategory of its domain on Cauchy-complete objects. 
The main result is that this reflector exists if and only if the relational doctrine has singleton objects and this happens if and only if its restriction to Cauchy-complete objects is equivalent to its completion with the rule of unique choice. 
We support our results with many examples, also falling outside the scope of standard doctrines, such as 
complete metric spaces, Banach spaces and compact Hausdorff spaces in the general context of monoidal topology, 
which are all shown to be Cauchy-complete objects for appropriate relational doctrines. 
\end{abstract}

% NOTE: it is good practice to \label all headings (and proclamations) immediately

% !TEX root = main.tex

\section{Introduction}
\label{sect:intro}

Cauchy-completeness is a concept coming from metric spaces: 
a metric space is said to be Cauchy-complete if every Cauchy-sequence converges to a limit point. 
In \cite{Lawvere73} Lawvere vastly generalised this notion, bringing it to the realm of enriched categories. 
Lawvere's key observation was that metric spaces are enriched categories over the quantale of extended non-negative real numbers and 
Cauchy-sequences in a space $X$ correspond to left adjoint bimodules from the one point space into $X$. 
Hence, he defined an enriched category \ct{B} to be Cauchy-complete when every left adjoint bimodule with codomain \ct{B}  is (derived from) an enriched functor into \ct{B}, recovering Cauchy-complete metric spaces as a special case. 
This notion, together with the associated Cauchy-completion, 
has been extensively studied and adapted in many different contexts 
\cite{Betti1982, Borceux1986, pjm/1102700481,Rosolini2000ANO, Street81}. 

From a logical standpoint, the notion of Cauchy-completeness resembles a form of \emph{choice rule}. 
If we regard bimodules as some sort of (binary) relations between enriched categories and enriched functors as  some sort of functions, 
then left adjoint bimodules are functional and total relations \cite{Street81} 
and so, Cauchy-completeness amounts to requiring that 
every functional and total relation is (the graph of) a function. 
This is precisely a formulation of the \emph{rule of unique choice}.
The purpose of this paper is to make this analogy precise, studying Cauchy-completeness and related constructions 
from a logical perspective. 

One of the most versatile tool to treat logic categorically are hyperdoctrines, or simply doctrines, introduced by Lawvere in \cite{Lawvere69,Lawvere70}. 
Doctrines give a functorial description of (theories) in predicate logic:
they are contravariant functors on a base category \ct{C}, which can be regarded as a category of contexts and substitutions, 
mapping each object to the poset of predicates over it ordered by logical entailment. 
The key feature of doctrines is that logical structures, such as connectives and quantifiers, become algebraic structures. 
This makes doctrines a simple and powerful framework, 
suitable to be adapted to the various logical systems one wants to work with and, 
indeed, it is not surprising that many variants appeared in the literature (see \cite{JacobsB:catltt,PittsCL} and references therein).

When doctrines have enough structure, \ie, conjunctions, equalities and existential quantifiers, 
one can easily formulate choice rules using this language (see for example \cite{Maietti-Rosolini16,MPR,Pep18}. 
So it might appear obvious that one can study Cauchy-completeness in the setting of doctrines, 
recovering also some motivating examples from enriched categories, such as metric spaces. 
However, doctrines do not recover this latter example in a natural way. In the definition of Cauchy-completeness, distances play the role of identity relations and 
since distances can be seen as equivalence relations rewritten with a monoidal operation, 
a naive attempt would be finding a doctrine over the category of metric spaces with a monoidal conjunction and where equality predicates are given by distances. 
This approach fails 
essentially because, as pioneered by Lawvere, equality predicates are given by left adjoints and, as shown in \cite{DagninoP22}, left adjoints necessarily produce trivial distances. 
% One of the main difficulty is that doctrines, modelling usual predicate logic, have to take care of variables, 
% and this is crucial for the equality predicates, which is mainly involved in deriving substitutions. 
% The problem is that, given two entailments $\phi(x)\otimes x=y\vdash \phi[y/x]$ and $\psi(x)\otimes x=y\vdash \psi[y/x]$ one has also $\phi(x)\otimes\psi(x)\otimes x=y\otimes x=y\vdash \phi[y/x]\otimes\psi[y/x]$ but not in general $\phi(x)\otimes\psi(x)\otimes x=y\vdash \phi[y/x]\otimes\psi[y/x]$ which is what one would expect from equality; forcing the last entailment to hold forces also the validity of the following equivalence $x=y\dashv\vdash x=y\otimes x=y$ that makes impossible the interpretation of the equality predicate as a (non-trivial) distance.

These difficulties arise from the fact that doctrines, abstracting predicate logic, have to take care of variables,
which is not trivial at all when dealing with a monoidal setting. % like metric spaces. 
On the other hand, one can not avoid the use of variables, 
as they are needed to model relations, 
which are nothing but predicates over contexts with many variables. 
To overcome this problem, 
in this paper we make the shift of working with \emph{relational doctrines} \cite{DagninoP23}, \ie, doctrines abstracting (a fragment of) the \emph{calculus of relations} \cite{Givant1,Peirce,Tarski41}. 
The calculus of relation  is a variable-free alternative to first order logic, where the primitive concepts are (binary) relations instead of (unary) predicates, together with some basic operations, such as relational identities, composition and converse. 
In general this calculus is less expressive than first order logic,\footnote{The calculus of relations is equivalent to first order logic with three variables \cite{Givant06}.}
but it is still quite expressive: 
for instance, it suffices to axiomatize set theory \cite{tarski1988formalization}. 
Thus, a relational doctrine on a category \ct{C}
is a contravariant functor mapping objects $(A,B)$ in the product of $\ct{C}$ with itself to a poset that collects relations from $A$ to $B$, 
together with transformations modelling relational identities, composition and converse. 

Given a relational doctrine, we define objects of the base that are Cauchy-complete as those that satisfy the rule of unique choice and we say that a doctrine satisfies the rule of unique choice when every object of its base is Cauchy-complete. 
We describe a 2-categorical construction that freely completes a relational doctrine with the rule of unique choice. 
This is obtained by adding left adjoint relations as new arrows in the base category, leaving relations unchanged. 
We also introduce the class of relational doctrines with singleton objects. 
These doctrines have 
the minimal structure needed to build the reflector of the full subcategory on Cauchy-complete objects, that is, the Cauchy-completion.
% Although the Cauchy completeness and the rule of unique choice are essentially the same concept, the related constructions apparently go in opposite directions, as the former produces a reflection of the subcategory of the base of a doctrine, while the latter  enlarges the base category with new arrows; the former gives an adjunction that involves the base category of a doctrine, while the latter gives a 2-adjunction in the 2-category of relational doctrines. It is then natural to ask how these two constructions match each other.
The main result is that Cauchy-complete objects make a reflective subcategory of the base if and only if the relational doctrine has singleton objects and this happens if and only if the restriction of the relational doctrine to Cauchy-complete objects is equivalent to the completion of the doctrine for the rule of unique choice.
We recover many examples as special cases:
beside the category of Cauchy-complete metric spaces,  we have 
the category of Banach spaces, which are the Cauchy-complete objects for a doctrine over the category of (semi-)normed vector spaces,
the category of compact Hausdorff spaces, which are the Cauchy-complete objects for an appropriate relational doctrine based on the category of topological spaces and 
the category of sheaves for a topology $j$ on an elementary topos \ct{E}, which are the Cauchy-complete objects for the relational doctrines of $j$-closed relations in \ct{E}. 
We also find as instance of our theorem, Walters' characterisation in \cite{W1} of sheaves over a frame as Cauchy-complete categories. 
Finally, we show how compact Hausdorff spaces in the general context of monoidal topology \cite{ClementinoH03,HofmannST14} can be characterised as Cauchy-complete objects for suitable relational doctrines constructed extending the Stone-Cech compactification to this setting. 

The rest of the paper is organised as follows. 
\cref{sect:rel-doc} recalls basic definitions and properties on relational doctrines and introduce some relevant examples.
In \cref{sect:ruleuc}, we introduce the rule of unique choice for relational doctrines, describing the universal construction that freely adds it to any relational doctrine. 
\cref{sect:singletons} defines relational doctrines with singleton  objects and proves our main results. 
Finally, in \cref{sect:compact} we discuss in detail the examples coming from monoidal topology.

% !TEX root = main.tex

\section{Preliminaries on relational doctrines}
\label{sect:rel-doc}

Lawvere's hyperdoctrines, or simply doctrines, introduced in \cite{Lawvere69,Lawvere70}, provide an algebraic approach to the study of syntactic theories and their models, relying on the intuition that a logic can be written as a functor mapping a context to the collections of formulas whose free variables are in that context. More specifically a \emph{doctrine} $\PDoc$ on a category $\CC$ is a contravariant functor \fun{\PDoc}{\CC\op}{\Pos}, 
where  \Pos denotes  the category of posets and monotone functions;
the category \CC is named the \emph{base} of the doctrine and, 
for $X$  in \CC, the poset $\PDoc(X)$ is called \emph{fibre over $X$}. For \fun{f}{X}{Y} an arrow in \CC, the monotone function \fun{\PDoc\reidx{f}}{\PDoc (Y)}{\PDoc (X)} is called \emph{reindexing along $f$}.
The base category consists of the objects of interest with their transformations (such as contexts and substitutions or sets and functions), 
a fibre $\PDoc(X)$ collects and orders the ``properties" of the object $X$ (such as predicates ordered by logical entailment or subsets ordered by inclusion) and 
reindexing allows to transport the properties between objects according to their transformations (such as substitution or inverse imaging).

In \cite{DagninoP23} \emph{relational doctrines} were introduced as a functorial description of the core fragment of the calculus of relations \cite{Tarski41}. 
Since binary relations can be seen as predicates over a pair of objects, relational doctrines are in particular functors of the form \fun{\RDoc}{\bop\CC}{\Pos}, where each fibre $\RDoc(X,Y)$ collects relations from $X$ to $Y$. 
The fragment of the calculus of relations that relational doctrines model is the one given by 
relational identities, composition and converse \cite{Givant1,Peirce,Tarski41},
% 
%For set-theoretic relations, 
%the identity relation on a set $X$ is the diagonal $\rid_X = \{\ple{x,x'}\in X \times X\mid x = x'\}$, 
%the composition of $\relr\in \Rel(X, Y)$ with $\rels\in \Rel(Y, Z)$ is the set $\relr\rcomp\rels = \{\ple{x,z} \in X\times Z \mid \ple{x,y}\in \relr,\, \ple{y,z}\in \rels \text{ for some }y\in Y\}$, and 
%the converse of $\relr\in\Rel(X,Y)$ is the set $\relr\rconv = \{\ple{y,x}\in Y\times X \mid \ple{x,y}\in \relr\}$. 
%These operations interact with reindexing, \ie inverse images, by the following inclusions: 
%$\rid_X \subseteq (f\times f)^{-1}(\rid_Y)$ and 
%$(f\times g)^{-1}(\relr)\rcomp (g\times h)^{-1}(\rels)\subseteq (f\times h)^{-1}(\relr\rcomp\rels)$ and also
%$((f\times g)^{-1}(\relr))\rconv \subseteq (g\times f)^{-1}(\relr\rconv)$. 
%The first two inclusions are not equalities in general: 
%the former is an equality when $f$ is injective, while the latter is an equality when $g$ is surjective. 
%% 
%These observations
leading to the following definition.

\begin{definition}[Relational Doctrine~\cite{DagninoP23}]\label[def]{def:rel-doc}
A \emph{relational doctrine} consists of: 
\begin{itemize}
\item a base category \CC, 
\item a functor \fun{\RDoc}{\bop\CC}{\Pos}, 
\item an element $\rid_X \in \RDoc(X,X)$, for every object $X$ in \CC, such that 
$\rid_X \order \RDoc\reidx{f,f}(\rid_Y)$, 
for every arrow \fun{f}{X}{Y} in \CC, 
\item a monotone function \fun{\blank\rcomp\blank}{\RDoc(X,Y)\times\RDoc(Y,Z)}{\RDoc(X,Z)}, for every triple of objects $X,Y,Z$ in \CC, such that 
$\RDoc\reidx{f,g}(\relr)\rcomp\RDoc\reidx{g,h}(\rels) \order \RDoc\reidx{f,h}(\relr\rcomp\rels)$, 
for all $\relr\in\RDoc(A,B)$, $\rels\in\RDoc(B,C)$ and 
\fun{f}{X}{A}, \fun{g}{Y}{B} and \fun{h}{Z}{C} arrows in \CC, 
\item a monotone function \fun{(\blank)\rconv}{\RDoc(X,Y)}{\RDoc(Y,X)}, for every pair of objects $X,Y$ in \CC,  such that 
$(\RDoc\reidx{f,g}(\relr))\rconv \order \RDoc\reidx{g,f}(\relr\rconv)$, 
for all $\relr\in\RDoc(A,B)$ and \fun{f}{X}{A} and \fun{g}{Y}{B}, 
\end{itemize}
satisfying the following equations for all 
$\relr\in\RDoc(X,Y)$, $\rels\in\RDoc(Y,Z)$ and $\relt\in\RDoc(Z,W)$
\begin{align*} 
\relr\rcomp(\rels\rcomp\relt) &= (\relr\rcomp\rels)\rcomp\relt 
& 
\rid_X\rcomp\relr &= \relr 
&
\relr\rcomp\rid_Y &= \relr 
\\
(\relr\rcomp\rels)\rconv &= \rels\rconv \rcomp \relr\rconv 
&
\rid_X\rconv &= \rid_X
&
\relr\rdconv &= \relr 
\end{align*} 
\end{definition}

The element $\rid_X$ is the \emph{identity} or \emph{diagonal} relation on $X$, 
$\relr\rcomp\rels$  is the \emph{relational composition} of $\relr$ followed by $\rels$, and 
$\relr\rconv$ is the \emph{converse} of the relation $\relr$.  
Note that all relational operations are lax natural transformations, but 
the operation of taking the converse, being  an involution, is actually strictly natural. 
The requirement of lax naturality becomes clear if one looks at the examples.  Take for instance the relevant one of set-theoretic relations, that can be organised into the relational doctrine \fun{\Rel}{\bop\Set}{\Pos} 
where $\Rel(X,Y)=\PP(X\times Y)$ and  $\Rel(f,g)=(f\times g)^{-1}$. Here it is an easy check that for $\relr\subseteq A\times B$ and $\rels\subseteq B\times C$ and for functions
\fun{f}{X}{A}, \fun{g}{Y}{B} and \fun{h}{Z}{C}, it holds that
$\Rel\reidx{f,g}(\relr)\rcomp\Rel\reidx{g,h}(\rels) =\{(x,z)\mid \exists_{y\in Y} (f(x),g(y))\in \relr\ \text{and}\ (g(y),h(z)\in \rels)\}$ and
$\Rel\reidx{f,h}(\relr\rcomp\rels)=\{(x,z)\mid \exists_{b\in B} (f(x),b)\in \relr\ \text{and}\ (b,h(z))\in\rels\}$;
thus the inclusion $\Rel\reidx{f,g}(\relr)\rcomp\Rel\reidx{g,h}(\rels)\subseteq\Rel\reidx{f,h}(\relr\rcomp\rels)$ is an equality if and only if $g$ is surjective. Similarly the inclusion $\rid_X =\{(x,x')\mid x=x'\}\subseteq \{(x,x')\mid f(x)=f(x')\}= \RDoc\reidx{f,f}(\rid_Y)$ is an equality if and only if $f$ is injective.

\begin{remark}\label[rem]{variabili}
Typically, doctrines modelling predicate logic are based on categories with finite products, which model concatenation of contexts, while product projections represent variables of a context. 
The variable-free nature of the calculus of relations is captured by the lack of  requirements on the base  category $\ct{C}$. A comparison between doctrines modelling predicate logic and relational doctrines is given in 
 \cite{DagninoP23}.
\end{remark}

\begin{remark}\label[rem]{rem:alternative-defs}
There are many alternative ways of defining relational doctrines. 
One can see them as certain internal dagger categories in a category of indexed posets (see \cite{DagninoP23}). 
Alternatively, they can be regarded as faithful framed bicategories \cite{Shulman08}, 
which are double categories with the additional structure of a fibration. 
In \cref{def:rel-doc}, we give a more explicit and elementary description of relational doctrines, 
to stay closer to the usual language of doctrines.
Extending all our results to the more general and proof-relevant setting of framed bicategories is an interesting problem we leave for future work. 
\end{remark}

\begin{example} \label[ex]{ex:rel-doc} 
\begin{enumerate}
\item\label{ex:rel-doc:vrel}
Let $\Qtl = \ple{\Car\Qtl,\qord,\qmul,\qone}$ be a commutative quantale. 
A \emph{$\Qtl$-relation}  \cite{HofmannST14} between sets $X$ and $Y$ is a function \fun{\alpha}{X\times Y}{\Car\Qtl}, where 
$\alpha(x,y)\in\Car\Qtl$ intuitively  measures how much elements $x$ and $y$ are related by $\alpha$. 
We consider 
the functor \fun{\VRel\Qtl}{\bop\Set}{\Pos} where 
$\VRel\Qtl(X,Y) = \Car\Qtl^{X\times Y}$ is the set of $\Qtl$-relations from $X$ to $Y$  with the pointwise order, 
$\VRel\Qtl\reidx{f,g}$ is precomposition with $f\times g$. 
The identity relation, relational composition and converse are defined as follows: 
\[
\rid_X(x,x') = \begin{cases}
\qone & x = x' \\
\bot  & x \ne x' 
\end{cases}
\qquad 
(\relr\rcomp\rels)(x,z) = \qsup_{y\in Y} (\relr(x,y)\qmul\rels(y,z))
\qquad 
\relr\rconv(y,x) = \relr(x,y) 
\]
where $\relr\in\VRel\Qtl(X,Y)$ and $\rels\in\VRel\Qtl(Y,Z)$. 
Special cases of this doctrine are
\fun{\Rel}{\bop\Set}{\Pos} , when the quantale is $\Bool = \ple{\{0,1\},\leq,\land,1}$, and 
metric relations, when the quantale is Lawvere's one $\RPos = \ple{[0,\infty],\ge,+,0}$ as in \cite{Lawvere73}.

\item\label{ex:rel-doc:met}
A metric space is a pair $X=\ple{\Car{X}, \dist_X}$ where $\fun{\dist_X}{\Car{X}\times \Car{X}}{[0,\infty]}$ is a distance, \ie  is such that $\dist_X(x,x)=0$, $\dist_X(x,x')=\dist_X(x',x)$ and $\dist_X(x,x')+\dist_X(x',x'')\geq \dist_X(x,x'')$. A non-expansive map $\fun{f}{X}{Y}$ is a function $\fun{f}{\Car{X}}{\Car{Y}}$ such that $\dist_X(x,x')\geq\dist_Y(f(x),f(x'))$. Metric spaces and non-expansive maps form the category $\Met$ which has a tensor product $X\otimes Y$ where $\Car{X\otimes Y}=\Car{X}\times\Car{Y}$ and $\dist_{X\otimes Y}(\ple{x,y},\ple{x',y'})=\dist_X(x,x')+\dist_Y(y,y')$. The functor $\fun{\MetDoc}{\bop{\Met}}{\Pos}$ maps $\ple{X,Y}$ to the poset of non-expansive maps $\Hom{\Met}{X\otimes Y}{\RPos}$ ordered point-wise (these functions are called bimodules in \cite{Lawvere73}); the action of $\MetDoc$ on a pair of non-expansive maps $\fun{\ple{f,g}}{\ple{A,B}}{\ple{X,Y}}$ and on  $\alpha$ in $\MetDoc(X,Y)$ is the composition $\alpha \circ (f\times g)$. The functor $\MetDoc$ is a relational doctrine: the identity relation, relational composition and converse are
defined as follows: 
\[
\rid_X = \dist_X
\qquad 
(\relr\rcomp\rels)(x,z) = \inf_{y \in Y} (\relr(x,y) + \rels(y,z)) 
\qquad 
\relr\rconv(y,x) = \relr(x,y)
\]
where $\relr\in\MetDoc(X,Y)$ and $\rels\in\MetDoc(Y,Z)$. 
\item\label{ex:rel-doc:vec}

Let $\ct{SVec}$ be the category of seminormed vector spaces over the field of real numbers and short linear maps.
Abusing with the notation, we identify  a (semi-)normed space with its underlying vector space and 
write $\Car X$ and $\Norm[X]\blank$  for the underlying set and the seminorm of the space $X$, respectively. 
Given semi-normed spaces $X$ and $Y$, denote by $X\times Y$ the product of $X$ and $Y$ and write $\Norm[X,Y]{\blank}$ for the semi-norm on it defined by 
$\Norm[X,Y]{\ple{\vecx,\vecy}} = \Norm[X]{\vecx} + \Norm[Y]{\vecy}$.  
The functor 
\fun{\SVecDoc}{\bop{\ct{SVec}}}{\Pos} sends $X,Y$ to the poset of semisemi-norms on $X\times Y$ bounded by $\Norm[X,Y]{\blank}$, i.e., 
functions \fun{\relr}{\Car{X}\times\Car{Y}}{[0,\infty]} such that 
\[ 
  \relr(\vecx,\vecy) + \relr(\vecx',\vecy') \ge  \relr(\vecx + \vecx', \vecy + \vecy') 
\qquad 
  \relr(a\vecx, a\vecy) = |a| \relr(\vecx,\vecy) 
\qquad 
  \Norm[X]{\vecx} + \Norm[Y]{\vecy} \ge \relr(\vecx,\vecy) 
\]
The order is the pointwise extension of the order of the Lawvere's quantale. 
The functor $\SVecDoc$ is a relational doctrine where 
%is analogous to $\VRel\RPos$, that is, 
\[
\rid_X(\vecx,\vecx') =  \Norm[X]{\vecx-\vecx'}  
\quad 
(\relr\rcomp\rels)(\vecx,\vecz)  = \inf_{\vecy\in\Car Y} (\relr(\vecx,\vecy) + \rels(\vecy,\vecz)) 
\quad 
\relr\rconv(\vecy,\vecx) = \relr(\vecx,\vecy)
\]
Concerning the axioms that these operations have to satisfy, we only check $\rid_X\rcomp\relr = \relr$, the other being immediate. 
First, from the fact that $\relr$ is bounded by $\Norm[X,Y]{\blank}$, we deduce that 
$ 
  \Norm[X]{\vecx-\vecx'} + \Norm[Y]{\vecy-\vecy'} 
    \ge \relr(\vecx-\vecx', \vecy-\vecy') 
    \ge \relr(\vecx,\vecy) - \relr(\vecx',\vecy') 
$, since $\relr$ is a semi-norm on $X\times Y$,  hence  we get 
$ 
  \Norm[X]{\vecx-\vecx'} + \relr(\vecx',\vecy') + \Norm[Y]{\vecy-\vecy'} \ge \relr(\vecx,\vecy)
$. This implies that 
$ 
  (\rid_X\rcomp\relr)(\vecx,\vecy) 
    \ge \rid_X(\vecx,\vecx') + \relr(\vecx',\vecy) 
    = \Norm[X]{\vecx-\vecx'} + \relr(\vecx',\vecy) + \Norm[Y]{\vecy-\vecy} 
    \ge \relr(\vecx,\vecy)
$, while the other inequality is straightforward.
\item\label{ex:rel-doc:walters}

Let $B$ be a locally partially ordered bicategory with objects in $B_0$ and arrows in $B_1$. Recall from \cite{Betti1982,W1} that a $B$-category $X$ is a triple $\ple{\Car{X}, e_X, d_X}$ of a set $\Car{X}$ together with functions $\fun{e_X}{\Car{X}}{B_0}$ and $\fun{d_X}{\Car{X}\times \Car{X}}{B_1}$ such that $\fun{d_X(x_1,x_2)}{e_X(x_1)}{e_X(x_2)}$ and 
\[
\id_{e_X(x)}\le d_X(x,x)\quad\quad d_X(x_2,x_3) \circ d_X(x_1,x_2)\le d_X(x_1,x_3)
\]
A $B$-functor $\fun{f}{X}{Y}$ is a function $\fun{f}{\Car{X}}{\Car{Y}}$ such that $e_Y(f(x))=e_X(x)$ and $d_X(x_1,x_2)\le d_Y(f(x_1),f(x_2))$. $B$-categories and $B$-functors form the category $\BCat{B}$. A $B$-bimodule (simply a bimodule) $\phi$ from $X$ to $Y$ is a function $\fun{\phi}{\Car{X}\times \Car{Y}}{B_1}$ such that $\fun{\phi(x,y)}{e_X(x)}{e_Y(y)}$ and 
\[
\phi(x,y)\circ d_X(x',x)\le \phi(x',y)\quad\quad d_Y(y,y')\circ \phi(x,y)\le \phi(x,y')
\]
A $B$-category $X$ is symmetric if $d_X(x_1,x_2)=d_X(x_2,x_1)$ and skeletal if $d_X(x_1,x_2)=e_X(x_1)=e_X(x_2)$ implies $x_1=x_2$. Denote by $\BCatss{B}$ the full subcategory of $\BCat{B}$ on symmetric skeletal $B$-categories.

We restrict our attention to the relevant case considered in \cite{W1}, that is when $B$ is $Rel(\ct{H})$, where $\ct{H}$ is a frame. Objects of $Rel(\ct{H})$ are those of $\ct{H}$, while an arrow from $u$ to $v$ is $w\le u\wedge v$ and the composition of two arrows is their meet.  $Rel(\ct{H})$-Bimodules define a relational doctrine 
$\fun{\BiMod}{\bop{\BCatss{Rel(\ct{H})}}}{\Pos}$, 
mapping symmetric and skeletal $Rel(\ct{H})$-categories $X$ and $Y$ to the bimodules from $X$ to $Y$. The identity relation, relational composition and converse are
%defined as follows: 
\[
\rid_{X} = d_X
\qquad 
(\phi\rcomp\psi)(x,z) = \bigvee_{y \in Y} (\phi(x,y) \wedge \psi(y,z)) 
\qquad 
\phi\rconv(y,x) = \phi(x,y)
\]

\item\label{ex:rel-doc:topos}
Let $\ct{E}$ be an elementary topos and denote by $\fun{\Sub}{\ct{E}\op}{\Pos}$ its subobject functor. A relation $r$ from $A$ to $B$ is an element of $\Sub(A\times B)$. As customary we freely confuse subobjects with any of their representatives. Let $r$ be the monic  $\fun{r}{X}{A\times B}$ and $s$ be  $\fun{s}{Y}{B\times C}$. Write $r_1$ and $r_2$ for the compositions of $r$ with projections $\fun{\pi_1}{A\times B}{A}$ and $\fun{\pi_2}{A\times B}{B}$ and similarly for $s_1$ and $s_2$. The relational composition $r\rcomp s$ in $\Sub(A\times C)$ is given by the following pullback
\[
r \rcomp s = \vcenter{\xymatrix@R=3ex@C=3ex{
&& W \ar[ld] \ar[rd] \ar@{}[dd]|{pb} && \\ 
& X \ar[ld]^-{r_1}\ar[rd]_-{r_2} && Y \ar[ld]^-{s_1} \ar[rd]_-{s_2} & \\ 
A && B && C 
}}
% \qquad 
%R\rconv = \vcenter{\xymatrix@C=3ex@R=3ex{
%X\ar[d]^{\ple{r_2,r_1}}  \\ 
%B\times A
%}}
%\qquad 
%\rid_X = \vcenter{\xymatrix@R=3ex@C=3ex{
%X \ar[d]^-{\Delta_X} \\ 
%X\times X
%}}
\]
The neutral elements of such a composition are the diagonal arrows, while the converse of $r$ is given by $\ple{\pi_2,\pi_1}^*r$, \ie the pullback of $r$ along $\ple{\pi_2,\pi_1}$. A Lawvere-Tierney topology over $\ct{E}$ (or simply a topology) is a natural transformation $\nat{j}{\Sub}{\Sub}$ such that for every subobject $\alpha$ in $\Sub(A)$ it holds that $\alpha \le j_A\alpha$ ($j$ is inflationary); $j_Aj_A\alpha \le j_A\alpha$ ($j$ is idempotent). A subobject $\alpha$ in $\Sub(A)$ is $j$-closed (or simply closed) if $j_A\alpha=\alpha$ \cite{Barr, MaclaneS:sheigl}. We will often omit subscripts from $j$ when these are clear. The subposet of $\Sub(A)$ on closed subobjects will be denoted by $\Sub_j(A)$. For every elementary topos $\ct{E}$ and every topology $j$ over $\ct{E}$ consider the functor $\fun{\SubDoc_j}{\bop{\ct{E}}}{\Pos}$ that maps $\fun{\ple{f,g}}{\ple{X,Y}}{\ple{A,B}}$ to $\fun{(f\times g)^*}{\Sub_j(A\times B)}{\Sub_j(X\times Y)}$. The functor $\SubDoc_j$ is a relational doctrine where
$\rid_X = j\Delta_X$ and  
$\relr \rcomp \rels = j(\relr \rcomp \rels)$ where the relational composition on the right is computed in $\Sub$; due to naturality of $j$, the converse of a closed relations is closed. 

\item\label{ex:rel-doc:kh}
Let $\Top$ be the category of topological spaces and continuous functions. For a space $X$ in $\Top$ denote by $\cl(X)$ the set of closed subsets of $X$ ordered by inclusion. We say that $R$ is a closed relation from the space $X$  to the space $Y$ if $R$ is an element  of $\cl(X\times Y)$. Among all topological spaces, the compact-Hausdorff ones have some pleasant properties with respect to the calculus of relations:  if $Y$ is compact,  $R\in \cl(X\times Y)$ and $S\in \cl(Y\times Z)$, then $R\rcomp S =\{\ple{x,z}\mid \exists_{y\in Y}\ \ple{x,y}\in R\ \text{and}\ \ple{y,z}\in S\}$ is a closed relation, \ie is in $\cl(X\times Z)$, whereas if $X$ and $Z$ are Hausdorff the diagonal relations, that are the neutral elements of the previous composition, are themselves closed. One can ``correct'' the lack of compact-Hausdorffness of a space via the Stone-Chech compactification, that provides the left adjoint $\beta$ to the full inclusion of compact-Hausdorff spaces into $\Top$. This suggests a way to build a relational doctrine of closed relations based on $\Top$. It is the functor $\fun{\TopDoc}{\bop\Top}{\Pos}$ that maps spaces $X$ and $Y$ to $\cl(\beta X \times \beta Y)$ and a pair of continuous functions $f,g$ to $(\beta f\times \beta g)^{-1}$.

\end{enumerate}
\end{example}

Relational doctrines are the objects of the 2-category \RDtn.

A 1-arrow \oneAr{F}{\RDoc}{\SDoc} in \RDtn from \fun{\RDoc}{\bop\CC}{\Pos} to \fun{\SDoc}{\bop\D}{\Pos}, 
is a pair \ple{\fn{F},\lift{F}} consisting of 
a functor \fun{\fn{F}}{\CC}{\D} and a natural transformation 
\nt{\lift{F}}{\RDoc}{\SDoc\circ \bop{\fn{F}}},

\[
\xymatrix@C=7.5em@R=1em{
{\bop\CC}\ar[rd]^(.4){\RDoc}_(.4){}="P"
\ar[dd]_{\bop{\fn{F}}}^{}="F"
&\\
 & {\ct{Pos}}\\
{\bop\D}\ar[ru]_(.4){\SDoc}^(.4){}="R"&
\ar"P";"R"_{\lift{F}\kern.5ex\cdot\kern-.5ex}="b"
}
\]
preserving relational identities, composition and converse, that is, satisfying 
$\rid_{\fn{F}X} = \lift{F}_{X,X}(\rid_X)$ and 
$\lift{F}_{X,Y}(\relr)\rcomp \lift{F}_{Y,Z}(\rels) = \lift{F}_{X,Z}(\relr\rcomp\rels)$ and 
$(\lift{F}_{X,Y}(\relr))\rconv = \lift{F}_{Y,X}(\relr\rconv)$, 
for $\relr\in\RDoc(X,Y)$ and $\rels\in\RDoc(Y,Z)$. 

A 2-arrow \twoAr{\theta}{F}{G} is a natural transformation \nt{\theta}{\fn{F}}{\fn{G}} such that 
$\lift{F}_{X,Y} \order \SDoc\reidx{\theta_X,\theta_Y}\circ \lift{G}_{X,Y}$, for all objects $X,Y$ in the base of $\RDoc$
\[
\xymatrix@C=18em@R=1.5em{
{\bop\CC}\ar[rd]^(.4){\RDoc}_(.4){}="P"
\ar@<-1ex>@/_/[dd]_{\bop{\fn{F}}}^{}="F"\ar@<1ex>@/^/[dd]^{\bop{\fn{F'}}}_{}="G"&\\
 & {\ct{Pos}}\\
{\bop\D}\ar[ru]_(.4){\SDoc}^(.4){}="R"&
\ar@/_/"P";"R"_{\lift{F}\kern.5ex\cdot\kern-.5ex}="b"
\ar@<1ex>@/^/"P";"R"^{\kern-.5ex\cdot\kern.5ex \lift{F'}}="c"
\ar"G";"F"_{.}^{\theta\op}\ar@{}"b";"c"|{\le}}
\]

%
%The condition of a 2-arrow \twoAr{\theta}{F}{G} is equivalent to both 
%$\lift{F}_{X,Y}(\relr) \order \gr{\theta_X}\rcomp \lift{G}_{X,Y}(\relr) \rcomp \gr{\theta_Y}\rconv$ and 
%$\lift{F}_{X,Y}(\relr) \rcomp \gr{\theta_Y} \order \gr{\theta_X}\rcomp \lift{G}_{X,Y}(\relr)$, for $\relr\in\RDoc(X,Y)$. 
 
In \cite{DagninoP22} also \emph{lax} 1-arrows are considered and 
examples of 1-arrows capturing the notions of  relation lifting, among which the Barr lifting \cite{Barr70}, are discussed there. 

We now report some basic facts about relational doctrines discussed in detail in \cite{DagninoP23}. 
Let us fix a relational doctrine \fun{\RDoc}{\bop{\CC}}{\Pos}.

\paragraph{Graphs} 
Every arrow \fun{f}{X}{Y} defines a relation $\gr{f}=\RDoc\reidx{f,\id_Y}(\rid_Y)\in\RDoc(X,Y)$ called the \emph{graph} of $f$. 
It is easy to see that the construction of graphs of arrows preserves composition and identities, that is, 
$\gr{g\circ f} = \gr{f}\rcomp\gr{g}$ and $\gr{\id_X} = \rid_X$. 
Furthermore, 1-arrows of \RDtn as defined above \emph{preserve graphs}, \ie, 
if \oneAr{F}{\RDoc}{\SDoc} is a 1-arrow and \fun{f}{X}{Y} is an arrow in \CC, we have 
$\lift{F}_{X,Y}(\gr{f}) = \gr{\fn{F}f}$ (note that $\gr{f}$ is a relation in $\RDoc$, while $\gr{\fn{F}f}$ is a relation in $\SDoc$). 

Relational composition allows us to express reindexing in relational terms and to show that it has a left adjoint, where a left adjoint in \Pos to a monotone function $\fun{g}{K}{H}$ is a monotone  function $\fun{f}{H}{K}$ such that for every $x$ in $K$ and $y$ in $H$, both $y\le gf(y)$ and $fg(x)\le x$ hold, or, equivalently, $y\le g(x)$ if and only if $f(y)\le x$. 
For \fun{f}{A}{X} and \fun{g}{B}{Y}  in \CC the reindexing map 
\fun{\RDoc\reidx{f,g}}{\RDoc(X,Y)}{\RDoc(A,B)} and its left adjoint 
\fun{\Ex^\RDoc\reidx{f,g}}{\RDoc(A,B)}{\RDoc(X,Y)} are given as follows: 
for $\relr\in\RDoc(X,Y)$ and $\rels\in\RDoc(A,B)$ 
\[
\RDoc\reidx{f,g}(\relr) = \gr{f} \rcomp \relr \rcomp \gr{g}\rconv 
\qquad 
\Ex^\RDoc\reidx{f,g}(\rels) = \gr{f}\rconv \rcomp \rels \rcomp \gr{g} 
\]
Thus, in the following, we will avoid the use of reindexing in calculations, replacing it by relational composition with graphs. 
Note that, since 1-arrows preserves graphs and left adjoints are defined in terms of graphs, they are preserved by 1-arrows. 

\paragraph{Extensional equality} 
Given parallel arrows \fun{f,g}{X}{Y} the condition $\gr{f}=\gr{g}$ defines an (external) equivalence relation $\exteq$ on the set $\Hom{\CC}{X}{Y}$.
We say that $f$ and $g$ are  \emph{$\RDoc$-equal}, or \emph{extensionally equal}, when  $f\exteq g$. 
The condition $\gr{f}=\gr{g}$  can be equivalently written as $\rid_X \order \RDoc\reidx{f,g}(\rid_Y)$, therefore, unfolding it when $\RDoc$ is $\Rel$, one find that  two functions $f$ and $g$ are $\Rel$-equal if for every $x$ and $y$ in the domain of $f$ and $g$, the equality $x = y$ implies $f(x) = g(y)$. 
In $\Rel$ this suffices to show that $f=g$, so $\Rel$-equal functions are actually equal, but for a general relational doctrine this need not hold.

\begin{definition} \label[def]{def:ext-eq}
Let \fun{\RDoc}{\bop\CC}{\Pos} be a relational doctrine. 
We say that the object $Y$ of $\CC$ is \emph{extensional} if 
for every \fun{f,g}{X}{Y}, $f \exteq g$ implies $f = g$. We say that the relational doctrine $\RDoc$ is  \emph{extensional} if every object of the base is extensional.
\end{definition}

\begin{example} \label[ex]{ex:extensional} 
\begin{enumerate}
\item\label{ex:extensional:vrel}
The relational doctrine 
 \fun{\VRel\Qtl}{\bop\Set}{\Pos} introduced in \refItem{ex:rel-doc}{vrel} is always extensional.
\item\label{ex:extensional:met}
Consider the relational doctrine $\fun{\MetDoc}{\bop{\Met}}{\Pos}$ described in \refItem{ex:rel-doc}{met}.  A metric space is extensional precisely when it is separated, \ie when $\delta_X(x,x')=0$ implies $x=x'$.
\item\label{ex:extensional:vec}
Consider the relational doctrine \fun{\SVecDoc}{\bop{\ct{SVec}}}{\Pos} described in \refItem{ex:rel-doc}{vec}. A space in $\ct{SVec}$ is extensional if the semi-norm is a norm, \ie if $\| \vecx \|=0$ implies $\vecx={\bf 0}$.

\item\label{ex:extensional:walters}
In the relational doctrine $\fun{\BiMod}{\bop{\BCatss{Rel(\ct{H})}}}{\Pos}$ in
\refItem{ex:rel-doc}{walters} the skeletality condition gives precisely the extensionality of $\BiMod$.

\item\label{ex:extensional:topos}
Let $j$ be  a topology over an elementary topos $\ct{E}$ and consider the relational doctrine $\SubDoc_j$ of \refItem{ex:extensional}{topos}. An object $A$ of $\ct{E}$ is extensional  precisely when $j\Delta_A=\Delta_A$, \ie when it is $j$-separated \cite{MaclaneS:sheigl}.

\item\label{ex:extensional:kh}
Consider the relational doctrine $\fun{\TopDoc}{\bop\Top}{\Pos}$ of \refItem{ex:extensional}{kh}. Here two continuous functions $f$ and $g$ are $\TopDoc$-equal precisely when $\beta f=\beta g$. The extensional topological spaces are the completely Hausdorff ones, \ie those spaces $X$ such that for every pair of different points $x\not =y$ there is a continuous $\fun{f}{X}{[0,1]}$ with $f(x)=0$ and $f(y)=1$ or, equivalently, those spaces $X$ such that  unite $\fun{\eta_X}{X}{\beta X}$ of the Stone-Cech compactification is injective. Indeed suppose $X$ is completely Hausdorff and let $\beta f=\beta g$. Since for every continuous function $\fun{f}{A}{X}$ it holds that $\beta f\circ \eta_A=\eta_X\circ f$, if $\eta_X$ is injective, then $f=g$. Conversely let $X$ be extensional, and take two points $x\not=y$ of $X$, \ie two different continuous functions $\fun{x,y}{1}{X}$. Since $\beta 1$ is $1$, the points $x,y$ determines two points $\fun{\beta x,\beta y}{1}{\beta X}$, that must be different as by extensionality of $X$ the equality $\beta x=\beta y$ would imply $x=y$. But $\beta x=\eta_X(x)$ and $\beta y=\eta_X(y)$ so $\eta_X(x)\not=\eta_X(y)$ showing that $\eta_X$ is injective.
\end{enumerate}
\end{example}

\paragraph{Duality}
Thanks to the existence of left adjoints along all arrows in \CC, we are able to define the \emph{opposite} of $\RDoc$, 
\ie, a relational doctrine $\RDoc\rop$ on the category $\CC\op$.
This is essentially a functor \fun{\RDoc\rop}{\CC\times\CC}{\Pos} defined by 
\[\xymatrix{
   \ple{X,Y} 
   \ar[dd]_-{\ple{f,g}} 
&& \RDoc(X,Y)\op 
   \ar[dd]^-{\Ex^\RDoc\reidx{f,g}} 
\\
   \qquad 
   \ar@{|->}[rr]^-{\RDoc\rop} 
&& \qquad 
\\
   \ple{A,B}
&& \RDoc(A,B)\op 
}\]
where the relational structure is the same as $\RDoc$. 
Note that, the fact that we take the opposite order in the fibres is essential to prove that reindexing of $\RDoc\rop$ laxly preserves the relational structure. 
This construction extends to a 2-functor 
\fun{(\blank)\rop}{\RDtn\co}{\RDtn}, which is an involution.

% !TEX root = main.tex

\section{The rule of unique choice}
\label{sect:ruleuc}

The rule of unique choice is a choice principle, weaker than the rule of choice, saying that a relation  that behaves like a function (in the sense that it relates every element of  its domain to a unique element in its codomain) is the graph of a function that for every $x$ in the domain picks the unique $y$ related to $x$. 
In the context of relational doctrines this can be given a formal definition, as we will see in this section. 

Let  $\RDoc$ be  a relational doctrine on $\CC$. 
We first identify those relations in $\RDoc(X,Y)$, generalising usual set-theoretic notions of relations that are functional (\ie if $y$ and $y'$ are related to the same $x$, then $y=y'$), total (\ie every $x$ is related to at least one $y$), injective (\ie different points in the domain are related to different points in the codomain) and surjective (\ie every $y$ in the codomain is related to at least one $x$ in the domain).
A relation $\relr$ in $\RDoc(X,Y)$ is said to be 
\begin{center}
\begin{tabular}{ll}
\emph{functional} if & $\relr\rconv \rcomp \relr \order \rid_Y$\\
\emph{total} if& $\rid_X \order \relr\rcomp\relr\rconv$ \\ 
\emph{injective} if & $\relr\rcomp\relr\rconv \order \rid_X$ \\ 
\emph{surjective} if & $\rid_Y \order \relr\rconv\rcomp\relr$ 
\end{tabular}
\end{center}
Note that $\relr$ is injective if and only if $\relr\rconv$ is functional and 
$\relr$ is surjective if and only if $\relr\rconv$ is total. 
Finally, $\relr$ is \emph{bijective} if both $\relr$ and $\relr\rconv$ are functional and total, that is,  if we have 
$\relr\rcomp\relr\rconv = \rid_X$ and $\relr\rconv\rcomp\relr = \rid_Y$. 

From the definition of total and functional relations, it follows immediately that  they are part of adjunctions, in the sense that, 
if $\relr$ is total and functional, then 
$\relr\rconv\rcomp\rels \order \relt$ if and only if $\rels \order \relr\rcomp\relt$ and 
$\rels\rcomp\relr \order \relt$ if and only if $\rels \order \relt\rcomp\relr\rconv$. 
Moreover, they form a discrete poset with respect to the order of the fibres of $\RDoc$.

\begin{proposition}[\cite{DagninoP23}]\label[prop]{prop:tot-fun-discrete}
Let $\relr$ and $\rels$ be total and functional relations in $\RDoc(X,Y$. Then, $\relr\order\rels$ implies $\relr = \rels$. 
\end{proposition}

It is also easy to see that graphs of arrows are always functional and total relations. 
We will say that an arrow $f$ in \CC  is \injective{\RDoc}, \surjective{\RDoc} or \bijective{\RDoc}, respectively, when its graph $\gr{f}$ is injective, surjective or bijective, respectively. 

\begin{proposition}\label[prop]{prop:arrows}
Let \fun{f}{A}{X} and \fun{g}{B}{Y} be arrows in \CC. 
The following hold: 
\begin{enumerate}
\item\label{prop:arrows:inj}
if $f$ and $g$ are \injective\RDoc, then $\RDoc\reidx{f,g} \circ \Ex^\RDoc\reidx{f,g} = \id_{\RDoc(A,B)}$; 
\item\label{prop:arrows:sur}
if $f$ and $g$ are \surjective\RDoc, then $\Ex^\RDoc\reidx{f,g} \circ\RDoc\reidx{f,g} = \id_{\RDoc(X,Y)}$; 
\item\label{prop:arrows:bij}
if $f$ and $g$ are \bijective\RDoc, then $\RDoc\reidx{f,g}$ is an isomorphism. 
\end{enumerate}
\end{proposition} 
\begin{proof}
We prove only \cref{prop:arrows:inj}, the other being analogous. 
Suppose that $f$ and $g$ are \injective\RDoc, \ie, $\gr{f}\rcomp\gr{f}\rconv = \rid_A$ and $\gr{g}\rcomp\gr{g}\rconv = \id_A$. 
For all $\relr\in\RDoc(A,B)$, we have 
$\RDoc\reidx{f,g}(\Ex^\RDoc\reidx{f,g}(\relr)) 
  = \gr{f}\rcomp\gr{f}\rconv\rcomp\relr\rcomp\gr{g}\rcomp\gr{g}\rconv 
  = \relr$, as needed. 
\end{proof}

\begin{proposition}\label[prop]{prop:arrows-iff}
Let \fun{f}{X}{Y} be an arrow in \CC. 
The following hold: 
\begin{enumerate}
\item\label{prop:arrows-iff:inj}
$f$ is \injective\RDoc if and only if $\RDoc\reidx{f,f}\circ \Ex^\RDoc\reidx{f,f} = \id_{\RDoc(X,X)}$; 
\item\label{prop:arrows-iff:sur}
$f$ is \surjective\RDoc if and only if $\Ex^\RDoc\reidx{f,f}\circ \RDoc\reidx{f,f} = \id_{\RDoc(Y,Y)}$; 
\item\label{prop:arrows-iff:bij}
$f$ is \bijective\RDoc if and only if $\RDoc\reidx{f,f}$ is  an isomorphism. 
\end{enumerate}
\end{proposition}
\begin{proof}
We prove only \cref{prop:arrows-iff:inj}, the other being analogous. 
The left-to-right implication follows by \refItem{prop:arrows}{inj}. 
For the other one, we have 
$\rid_X = \RDoc\reidx{f,f}(\Ex^\RDoc\reidx{f,f}(\rid_A)) = \gr{f}\rcomp\gr{f}\rconv\rcomp\gr{f}\rcomp\gr{f}\rconv$. 
Then, sing $\gr{f}$ is total, we get 
$\gr{f}\rcomp\gr{f}\rconv 
  \order  \gr{f}\rcomp\gr{f}\rconv\rcomp\gr{f}\rcomp\gr{f}\rconv
  = \rid_X$, 
proving that $f$ is  \injective\RDoc. 
\end{proof}

The following corollary is straightforward. 

\begin{corollary}\label[cor]{cor:bij}
Let \fun{f}{X}{Y} be an isomorphism of \CC, then $f$ is \bijective\RDoc. 
\end{corollary}
\begin{proof}
Immediate by \refItem{prop:arrows-iff}{bij}, as $\RDoc\reidx{f,f}$ is an isomorphism. 
\end{proof}

The converse of \cref{cor:bij} does not hold in general, 
unless, as we will see, the domain of the arrow satisfies the rule of unique choice. This shows also that the rule of unique choice can be seen as a form of balancedness of the base category.

A functional and total relation $\relr\in\RDoc(X,Y)$
has a \emph{tracking arrow} if there is \fun{f}{X}{Y} in \CC with $\gr{f}= \relr$.
Note that, by \cref{prop:tot-fun-discrete}, it suffices to require that $\gr{f} \order \relr$.

\begin{definition}\label[def]{def:ruc}
Let \fun{\RDoc}{\bop\CC}{\Pos} be a relational doctrine. 
An object $Y$ in $\CC$ satisfies the rule of unique choice, \RUC for short, if for every $X$ in $\CC$ every functional and total relation in $\RDoc(X,Y)$
has a tracking arrow. The object $Y$ satisfies the strong rule of unique choice, \SRUC for short, if for every $X$ in $\CC$ every functional and total relation in $\RDoc(X,Y)$
has a unique tracking arrow. A relational doctrine satisfies \RUC (respectively \SRUC) if every object of its base satisfies \RUC (respectively \SRUC).
\end{definition}

Functional and total relations are left adjoints, so objects  satisfying  \RUC enjoy a form of ``completeness'' that closely recalls the Cauchy-completeness, for this reason we will often call these objects \emph{Cauchy-complete}. Analogously an object $Y$ that satisfies \SRUC will be also called \emph{strongly Cauchy-complete}.

\begin{proposition}\label[prop]{prop:sruc-vs-ruc}
An object $Y$ strongly Cauchy-complete if and only if it is Cauchy-complete and extensional.
\end{proposition}
\begin{proof} 
If $Y$ is strongly Cauchy-complete, then it is obviously Cauchy-complete. 
Moreover any arrow $\fun{f}{X}{Y}$ is the unique tracking arrow of its graph $\gr{f}$. 
Therefore, if $\gr{f}=\gr{g}$, then $g$ tracks $\gr{f}$, hence $f = g$, as needed. 
Conversely, suppose $Y$ is extensional and Cauchy-complete. 
Then, given a total and functional relation $\relr$ in $\RDoc(X,Y)$ there is $f$ such that $\gr{f} = \relr$ and, 
if $g$ is another arrow tracking $\relr$, we have 
$\gr{g}=\relr=\gr{f}$, hence, by extensionality we get $f = g$.
\end{proof}

The following proposition shows that when the (strong) rule of unique choice holds, the converse of \cref{cor:bij} holds as well. 

\begin{proposition}\label[prop]{prop:sruc-iso}
Let \fun{f}{X}{Y} be an arrow in \CC where $X$ and $Y$ are extensional objects. 
If $X$ is Cauchy-complete and $f$ is \bijective\RDoc then $f$ is an isomorphism. 
\end{proposition} 
\begin{proof}
If $f$ is \bijective\RDoc, $\gr{f}\rconv$ is a functional and total relation in $\RDoc(Y,X)$ so there is a unique arrow \fun{g}{Y}{X} such that $\gr{g} = \gr{f}\rconv$. 
Then, we have 
$\gr{f\circ g} = \gr{f}\rconv\rcomp\gr{f} = \rid_Y$ and 
$\gr{g\circ f} = \gr{f}\rcomp\gr{f}\rconv = \rid_X$.  Extensionality of $X$ and $Y$ imply 
$f\circ g = \id_Y$ and $g\circ f = \id_X$. 
\end{proof}

\begin{example} \label[ex]{ex:ruc-doc} 
\begin{enumerate}
\item\label{ex:ruc-doc:vrel}
The relational doctrine \fun{\VRel\Qtl}{\bop\Set}{\Pos} of \refItem{ex:rel-doc}{vrel} need not satisfy \RUC. It is shown in \cite{HofmannST14} (proposition III.1.2.1)  that if the quantale $\Qtl$ is affine, \ie if $\qone$ is the top element, then 
$\VRel\Qtl$ satisfies \RUC if and only if the quantale is lean, \ie if $x\lor y=\top$ and $x\wedge y=\bot$ implies that $x=\top$ or $y=\top$.

\item\label{ex:ruc-doc:met}
The relational doctrine $\fun{\MetDoc}{\bop{\Met}}{\Pos}$ of \refItem{ex:rel-doc}{met} does not satisfy \RUC as a metric space is Cauchy-complete if and only if it is complete in the standard sense. 
The proof basically follows the original arguments given by Lawvere's in \cite{Lawvere73}. 
Rewritten in the language of relational doctrines, Lawvere's theorem says that a metric space $Y$ is complete if and only if for every $X$ and every $\relr$ in $\MetDoc(X,Y)$ which is a left adjoint, \ie such that there is $\rels$ in $\MetDoc(Y,X)$ with $\rels\rcomp\relr\le \dist_Y$ and $\dist_X\le \relr\rcomp\rels$,  there is a non-expansive map $\fun{f}{X}{Y}$ such that $\relr(x,y)=\dist_Y(f(x),y)$. 
The condition of Cauchy-completeness formulated in \cref{def:ruc} is a slightly weaker as it asks the right adjoint to be $\relr\rconv$ which is not always the case, in other words, not all left adjoint relations are functional and total, while all functional and total relations are left adjoints. 
Nevertheless Lawvere's argument can still be carried out. 
The only care to be taken is in proving the necessary condition. 
Suppose $Y$ is Cauchy-complete and take a Cauchy sequence $(y_n)_{n\in\N}$ in it. 
Lawvere defined a left adjoint relation $\relr$ in $\MetDoc(1,Y)$, \ie from the one-point space to $Y$, as 
$\relr(\star,y) = \lim\limits_{n\rightarrow\infty}\dist_Y(y_n,y)$, which is well defined since $(\dist_Y(y_n,y))_{n\in\N}$ is a Cauchy sequence in $[0,\infty]$. 
However, it is easy to see that $\relr$ is actually functional and total,  
hence there is $\fun{l}{1}{Y}$ such that $\relr(\star,y) = \lim\limits_{n\rightarrow\infty}\dist_Y(y_n,y) = \dist_Y(l(\star),y) = \gr{l}(\star,y)$. 
Taking $y = l(\star)$, we get $\lim\limits_{n\rightarrow\infty} \dist_Y(y_n,l(\star)) = \dist_Y(l(\star),l(\star)) = 0$, proving that 
$l(\star)$ is a limit of $(y_n)_{n\in\N}$. 

\item\label{ex:ruc-doc:vec}
The relational doctrine \fun{\SVecDoc}{\bop{\ct{SVec}}}{\Pos} of \refItem{ex:rel-doc}{vec} does not satisfies \RUC as a semi-normed vector space is Cauchy-complete if and only if it is complete as a metric space with the distance induced by the semi-norm.
The argument is essentially the same as Lawvere's one for metric spaces, with some differences due to the fact that we have to take into account the additional vector space structure, in particular we use the axiom of choice to ensure that every vector space has a basis.
For this reason we will carry out it in detail. 
Let us start by recalling a couple of easy properties of total and functional relations in $\SVecDoc$. 
If $\relr$ is a total and functional relation in $\SVecDoc(X,Y)$, the the following (in)equalities hold, for all $\vecx\in\Car{X}$ and $\vecy\in\Car{Y}$: 
$\relr(\veczero,\vecy) = \Norm[Y]{\vecy}$ and 
$\Norm[X]\vecx + \relr(\vecx,\vecy) \ge \Norm[Y]\vecy$. 

Let $X,Y$ be semi-normed vector spaces with $Y$ complete and $\relr$ a total and functional relation in $\SVecDoc(X,Y)$. 
Hence, for every $\vecx$ in $X$, we have $0 = \Norm[X]{\vecx-\vecx} \ge \inf_{\vecy\in \Car{Y}} \relr(\vecx,\vecy) + \relr\rconv(\vecy,\vecx)$, as $\relr$ is total, 
and this implies that $\inf_{\vecy\in\Car{Y}} \relr(\vecx,\vecy) = 0$. 
Therefore, for every $n\in\N$, there is $\vecy_{\vecx,n}\in\Car{Y}$ such that $1/n > \relr(\vecx,\vecy_{\vecx,n})$.
The sequence $(\vecy_{\vecx,n})_{n\in\N}$ is a Cauchy-sequence: 
for all $n\in\N$ and $h,k\ge 2n$, we have 
$1/n \ge 1/h + 1/k > \relr\rconv(\vecy_{\vecx,h},\vecx) + \relr(\vecx,\vecy_{\vecx,k}) \ge \Norm[Y]{\vecy_{\vecx,h} - \vecy_{\vecx,k}}$, 
as $\relr$ is functional. 
Let $\{ \vecx_i \mid i \in I \}$ be a basis for $X$ (whose existence relies on the axiom of choice)
and denote by $\vecy_i$ the limit of the sequence $(\vecy_{\vecx_i,n})_{n\in\N}$. 
We get a linear map \fun{f}{X}{Y}, which is the unique such that $f(\vecx_i) = \vecy_i$. 
To show that $f$ is indeed an arrow in \ct{SVec}, we have to prove that $\Norm[X]\vecx \ge \Norm[Y]{f(\vecx)}$, for all $\vecx\in\Car{X}$, but, 
thanks to properties of semi-norms, it suffices to check that 
$\Norm[X]{\vecx_i} \ge \Norm[Y]{f(\vecx_i)}$, for all $i\in I$. 
Since $\relr$ is total and functional, we have 
$\Norm[X]{\vecx_i} + \relr(\vecx_i,f(\vecx_i)) \ge \Norm[Y]{f(\vecx_i)}$, for all $i \in I$. 
Note that $\relr(\vecx_i,f(\vecx_i)) = \relr(\vecx_i,\vecy_i) =  0$, since  
$1/n + \Norm[Y]{\vecy_{\vecx_i,n} - \vecy_i} 
  > \relr(\vecx_i,\vecy_{\vecx_i,n}) + \Norm[Y]{\vecy_{\vecx_i,n} - \vecy_i} 
  \ge \relr(\vecx_i,\vecy_i) \ge 0$ and 
$\lim\limits_{n\rightarrow\infty} 1/n + \Norm[Y]{\vecy_{\vecx_i,n} - \vecy_i} = 0$, as $\lim\limits_{n\rightarrow\infty} \vecy_{\vecx_i,n} = \vecy_i$. 
Thus, we get $\Norm[X]{\vecx_i} \ge \Norm[Y]{f(\vecx_i)}$ for all $i \in I$, as needed. 
To conclude that $f$ tracks $\relr$, it suffices to show that 
$\Norm[Y]{f(\vecx_i) - \vecy_i} \ge \relr(\vecx_i,\vecy_i)$, for all $i\in I$, 
thanks to the fact that $\relr$ is a semi-norm on $X\times Y$. 
Since $\relr(\vecx_i,f(\vecx_i)) = 0$,  we deduce 
$\Norm[Y]{f(\vecx_i) - \vecy} 
  = \relr(\vecx_i,f(\vecx_i) + \Norm[Y]{f(\vecx_i) - \vecy} 
  \ge \relr(\vecx_i,\vecy)$, as needed. 

Conversely, let $Y$ be a Cauchy-complete vector space and let $(\vecy_n)_{n\in\N}$ be a Cauchy-sequence in $Y$. 
For every real number $a$ and $\vecy\in\Car{Y}$, the sequence 
$(\Norm[Y]{a\vecy_n - \vecy})_{n\in\N}$ is a Cauchy-sequence in $[0,\infty]$: 
this is an immediate consequence of the fact that $(\vecy_n)_{n\in\N}$ is a Cauchy-sequence in $Y$, observing that 
for all $h,k\in\N$, $\abs{a} \Norm[Y]{\vecy_h - \vecy_k} \ge \abs{\Norm[Y]{a\vecy_h - \vecy} - \Norm[Y]{a\vecy_k - \vecy}}$. 
Let $\R_\infty$ denote the vector sapce of real numbers with the (semi-)norm $\Norm[\R_\infty]{\blank}$ given by 
$\Norm[\R_\infty]{0} = 0$ and $\Norm[\R_\infty]{a} = \infty$ if $a\ne 0$. 
Then, we define a function \fun{\relr}{\Car{\R_\infty}\times\Car{Y}}{[0,\infty]} as follows: 
$\relr(a,\vecy) = \lim\limits_{n\rightarrow\infty} \Norm[Y]{a\vecy_n - \vecy}$. 
It is easy to see that $\relr$ is a relation in $\SVecDoc(\R_\infty,Y)$. 
It is total, observing that, since $(\vecy_n)_{n\in\N}$ is a Cauchy sequence, 
for every $\epsilon > 0$, there is $n_\epsilon\in\N$, such that, for all $k\ge n_\epsilon$, $\epsilon > \Norm[Y]{\vecy_{n_\epsilon} - \vecy_k}$, 
which implies $\abs{a} \epsilon \ge \lim\limits_{n\rightarrow\infty} \Norm[Y]{a\vecy_n-a\vecy_{n_\epsilon}} = \relr(a,a\vecy_{n_\epsilon})$. 
Finally, it is functional, noting that 
$\relr\rconv(\vecy,a) + \relr(a,\vecy') 
  =\lim\limits_{n\rightarrow\infty} \Norm[Y]{a\vecy_n - \vecy} + \Norm[Y]{a\vecy_n - \vecy'} 
  \ge \lim\limits_{n\rightarrow\infty} \Norm[Y]{\vecy-\vecy'} 
  = \Norm[Y]{\vecy-\vecy'}$. 
Since $Y$ is Cauchy-complete, we get an arrow \fun{f}{\R_\infty}{Y} such that, for all $a\in\Car{\R_\infty}$ and $\vecy\in\Car{Y}$, 
we have $\relr(a,\vecy) = \Norm[Y]{f(a) - \vecy}$. 
Therefore, taking $l = f(1)$, we get 
$\relr(1,l) = 0$, that is, 
$\lim\limits_{n\rightarrow\infty}\Norm[Y]{\vecy_n - l} = 0$, proving that $l$ is a limit of $(\vecy_n)_{n\in\N}$, as needed. 

Finally, recalling from \refItem{ex:extensional}{vec} that the extensional objects in $\SVecDoc$ are normed vector spaces, 
using \cref{prop:sruc-vs-ruc}, we get that 
an object $Y$ in \ct{SVec} is strongly Cauchy-complete if and only if it is a Banach space.

\item\label{ex:ruc-doc:walters}

The relational doctrine $\fun{\BiMod}{\bop{\BCatss{Rel(\ct{H})}}}{\Pos}$ introduced in \refItem{ex:rel-doc}{walters} is extensional (see \refItem{ex:extensional}{walters}). So strongly Cauchy-complete objects coincides with Cauchy-complete ones. Since composition of relations in $\BiMod$ is defined using meets and suprema, a bimodules $\phi$ is a left adjoint if and only if its right adjoint is $\phi\rconv$, so Cauchy-complete objects of $\BCatss{Rel(\ct{H})}$ are the Cauchy-complete symmetric and skeletal $Rel(\ct{H})$-categories as in \cite{W1}.

\item\label{ex:ruc-doc:topos}
Let $\ct{E}$ be an elementary topos, $j$ a topology over it and consider the functor $\fun{\SubDoc_j}{\bop{\ct{E}}}{\Pos}$ of \refItem{ex:rel-doc}{topos}. 
An object of $\ct{E}$ satisfies is strongly Cauchy-complete if and only if it is a $j$-sheaf. Indeed, suppose $Y$ is strongly Cauchy-complete (so $Y$ is extensional by \cref{prop:sruc-vs-ruc}, hence $j$-separated) and consider the $j$-dense arrow $\fun{\eta_Y}{Y}{s(Y)}$ to its associated $j$-sheaf $s(Y)$. The relation $\gr{\eta_Y}\rconv$ is functional and total, so it is the graph of an arrow which is necessarily the inverse of $\eta_Y$. To prove the converse take a functional and total $\alpha$ in $\SubDoc_j(X,Y)\subseteq \Sub(X,Y)$. If $Y$ is a $j$-sheaf, then $\alpha$ is functional an total also in $\Sub(X,Y)=\Sub(X\times Y)$. Since subobjects doctrines satisfies the \SRUC \cite{JacobsB:catltt}, there a unique is $\fun{f}{X}{Y}$ with $\alpha=(f\times\id_Y)^*\Delta_Y=(f\times\id_Y)^*j\Delta_Y=\gr{f}$ where $\gr{f}$ is computed with respect to the doctrine $\SubDoc_j$.

\item\label{ex:ruc-doc:kh}
Consider the doctrine \fun{\TopDoc}{\bop\Top}{\Pos} of \refItem{ex:rel-doc}{kh}. 
A topological space is strongly Cauchy-complete if and only if it is compact and Hausdorff. Indeed suppose $Y$ is compact and Hausdorff: a functional and total relation $F$ in $\TopDoc(X,Y)=\cl(\beta X\times \beta Y)$ is a continuous functions $\fun{F}{\beta X}{\beta Y}$, whose tracking arrow is given composing $F$ with $\fun{\eta_X}{X}{\beta X}$ and $\fun{\eta_Y^{-1}}{\beta Y}{Y}$ which is a homeomorphis as $Y$ is compact and Hausdorff. Conversely suppose $Y$ is strongly Cauchy-complete, and consider $\fun{\eta_Y}{Y}{\beta Y}$. It is $\gr{\eta_Y}\rcomp\gr{\eta_Y}\rconv = \{(x,z)\in \beta Y\times \beta Y\mid \exists_{y\in \beta^2 Y} \beta(\eta_Y)(x)=y\ \text{and}\ y=\beta(\eta_Y)(z)\}=\{(x,z)\mid \beta(\eta_Y)(x)=\beta(\eta_Y)(z)\} =\rid_Y$ so $\gr{\eta_Y}\rconv$ is functional. For totality consider $\gr{\eta_Y}\rconv\rcomp\gr{\eta_Y}=\{(u,v)\in \beta^2Y\times \beta^2Y\mid \exists_{w\in \beta Y} u=\beta(\eta_Y)(w)\ \text{and}\ \beta(\eta_Y)(w)=v \}$. Since $\fun{\beta(\eta_Y)}{\beta Y}{\beta^2Y}$ is an homeomorphis $\gr{\eta_Y}\rconv\rcomp\gr{\eta_Y}=\{(u,v)\mid u=v\}=\rid_{\beta Y}$, then $\gr{\eta_Y}\rconv$ in $\TopDoc(\beta Y, Y)$ has a unique tracking arrow $\fun{g}{\beta Y}{Y}$ which is the inverse of $\eta_Y$ making $Y$ compact-Hausdorff.
\end{enumerate}
\end{example}

\begin{remark}\label[rem]{rem:cauchy}
\cref{ex:ruc-doc} shows how the rule of unique choice in a relational doctrines is linked to the notion of Cauchy completeness for some enriched categories, such as metric spaces.
However, doctrines, being a proof-irrelevant setting, can only deal with simple examples of enriched categories. 
In order to cope with them in full generality, one should adopt a proof-relevant framework like framed bicategories (\cf \cref{rem:alternative-defs}). 
\end{remark}

The rest of this section is devoted to  defining 
constructions which produce relational doctrines satisfying \SRUC out of any relational doctrine. 
We focus on \SRUC rather than on \RUC as this allows us to show universal properties of the presented constructions
which would not hold if we considered just \RUC (see \cref{thm:doc-to-ruc}). 
Denote by $\RDtnRuc$ be the 1-full and 2-full subcategory of $\RDtn$ on relational doctrines that satisfy  \SRUC. 

Let $\RDoc$ be a relational doctrine over $\CC$. 
A first naive approach to build a relational doctrine satisfying \SRUC starting from $\RDoc$ is  
to consider the restriction $\complete{\RDoc}$ of $\RDoc$ to the full subcategory $\complete[\RDoc]{\CC}$ of \CC spanned by 
strongly Cauchy-complete objects, \ie, the composite 
\[\xymatrix{
\bop{\complete[\RDoc]{\CC}} \ar@{^(->}[r] & \bop\CC \ar[r]^-{\RDoc} & \Pos 
}\]
Clearly, $\complete{\RDoc}$ satisfies \SRUC (by \cref{prop:sruc-vs-ruc}) since the relational structure is exactly that of $\RDoc$. 
However simple, this construction has the major drawback of being not functorial. 
Indeed, 1-arrows $F$ in \RDtn do not need to preserve Cauchy-complete objects, essentially because $\fn{F}$ and $\lift{F}$ need not be full and componentwise surjective, respectively.  

In order to recover functoriality, we now define a different construction. 
First of all, we observe that 
if $\relr$ in $\RDoc(A,B)$ and $\rels$ in $\RDoc(B,C)$ are functional so is $\relr\rcomp\rels$ in $\RDoc(A,C)$ as $(\relr\rcomp\rels)\rconv \rcomp \relr\rcomp\rels= \rels\rconv\rcomp\relr\rconv \rcomp \relr\rcomp\rels\order  
\rels\rconv\rcomp\rid_B\rcomp\rels=  
\rels\rconv\rcomp\rels\order\rid_C$. Similarly, if $\relr$ and $\rels$ are total, so is $\relr\rcomp\rels$.
This defines a category $\Map\RDoc$ whose objects are those of $\CC$ and arrows from $A$ to $B$ are functional and total relations in $\RDoc(A,B)$.
Then, for every \fun{\relr}{X}{A} and \fun{\rels}{Y}{B} in $\Map\RDoc$, 
we let $\Ruc\RDoc(A,B) = \RDoc(A,B)$ and 
$\fun{{\Ruc\RDoc}\reidx{\relr,\rels}}{\Ruc\RDoc(A,B)}{\Ruc\RDoc(X,Y)}$ be the function $\relt \mapsto \relr\rcomp\relt\rcomp\rels\rconv$. 
These assignments determine a functor $\fun{\Ruc{\RDoc}}{\bop{\Map\RDoc}}{\Pos}$, 
which is also a relational doctrines where relational identities, composition and  converse are those of $\RDoc$.
Note that we can define a 1-arrow \oneAr{\GrAr[\RDoc]}{\RDoc}{\Ruc\RDoc} in \RDtn as depicted below 
\[
\xymatrix@C=7.5em@R=1em{
{\bop\CC}\ar[rd]^(.4){\RDoc}_(.4){}="P"
\ar[dd]_{\bop{\fn\GrAr[\RDoc]}}^{}="F"
&\\
 & {\ct{Pos}}\\
{\bop{\Map{\RDoc}}}\ar[ru]_(.4){\ \ \Ruc\RDoc}^(.4){}="R"&
\ar"P";"R"_{\lift\GrAr[\RDoc]\kern.5ex\cdot\kern-.5ex}="b"
}
\]
where $\fun{\fn\GrAr[\RDoc]}{\CC}{\Map\RDoc}$ is the identity on objects and it maps an arrow $f$ to its graph $\gr{f}$, 
and \nt{\lift\GrAr[\RDoc]}{\RDoc}{\Ruc\RDoc\bop{\fn\GrAr[\RDoc]}} is componentwise the identity. 
We will often omit the subscript from $\Ruc\RDoc$ when it is clear from the context. 

\begin{proposition}\label[prop]{prop:sruc}
For a relational doctrine \fun{\RDoc}{\bop\CC}{\Pos} the following are equivalent:
\begin{enumerate} 
\item $\RDoc$ satisfies \SRUC 
\item \oneAr{\GrAr[\RDoc]}{\RDoc}{\Ruc\RDoc} is an isomorphism in \RDtn 
\item \oneAr{\GrAr[\RDoc]}{\RDoc}{\Ruc\RDoc} is a section in \RDtn 
\end{enumerate} 
\end{proposition}
\begin{proof}
It suffices to reason on $\fn\GrAr$ as $\lift\GrAr$ is always componentwise an identity. 
$(1)\Rightarrow(2)$. 
If $\RDoc$ satisfies \SRUC, then $\fn\GrAr$ is fully faithful hence, as it is the identity on objects, it is an isomorphism. 
$(2)\Rightarrow(3)$. It is trivial. 
$(3)\Rightarrow(1)$. 
Let \oneAr{F}{\Ruc\RDoc}{\RDoc} be the retraction of $\GrAr[\RDoc]$ and $\relr$ a total  and functional relation in  $\RDoc(X,Y)$. 
Then, \fun{\relr}{X}{Y} is an arrow in $\Map\RDoc$ and so \fun{f = \fn{F}\relr}{\fn{F}X}{\fn{F}Y} is an arrow in \CC. 
Since $\lift{F}$ preserves graphs and $\gr{\relr} = \relr$, we get 
$\lift{F}_{X,Y}(\relr) = \lift{F}_{X,Y}(\gr{\relr}) = \gr{f}$. 
From $F\circ\GrAr[\RDoc] = \Id_\RDoc$ and $\fn{\GrAr[\RDoc]}$ is the identity on objects and $\lift{\GrAr[\RDoc]}$ is componentwise the identity, we get  that 
$\fn{F}$ is the identity on objects and $\lift{F}$ is componentwise the identity. 
Hence, we conclude $\relr = \lift{F}_{X,Y}(\relr) = \gr{f}$, proving that $f$ tracks $\relr$ and so $\RDoc$ satisfies \SRUC. 
\end{proof}
An immediate consequence of the  proposition above is that $\Ruc\RDoc$ satisfies \SRUC, since, due to the idempotency of the construction, 
$\GrAr[\Ruc\RDoc]$ is an identity. 

The following proposition shows that the construction of $\Ruc\RDoc$ is universal, 
proving that it determines a left 2-adjoint of the inclusion $\RDtnRuc \hookrightarrow \RDtn$. 

\begin{theorem}\label[thm]{thm:doc-to-ruc} 
Let \fun\RDoc{\bop\CC}{\Pos} be a relational doctrine and 
\fun\SDoc{\bop\D}{\Pos} a relational doctrine satisfying \SRUC. 
The functor 
\[
\fun{\blank \circ \GrAr[\RDoc]}{\RDtnRuc(\Ruc\RDoc,\SDoc)}{\RDtn(\RDoc,\SDoc)} 
\]
is an isomorphism of categories. 
\end{theorem}
\begin{proof} 
First of all, we prove that it is an isomorphism on objects. 
Let $\oneAr{F}{\RDoc}{\SDoc}$ be a 1-arrow in \RDtn.  
Every functional and total relation $\relr$ in $\RDoc(A,B)$ determines a functional and total relation $\lift{F}(\relr)$ in $\SDoc(\fn{F}A,\fn{F}B)$. 
By \SRUC, we get that $\lift{F}(\relr)=\gr{e_\relr}$ for a unique tracking arrow $\fun{e_\relr}{\fn{F}A}{\fn{F}B}$. 
Define a 1-arrow \oneAr{F'}{\Ruc\RDoc}{\SDoc} as follows: 
\begin{itemize}
\item the functor \fun{\fn{F'}}{\Map\RDoc}{\D} is given by $\fn{F'}A = \fn{F}A$ and $\fn{F}\relr = e_\relr$; 
\item the natural transformation \nt{\lift{F'}}{\Ruc\RDoc}{\SDoc\bop{\fn{F'}}} is given by 
$\lift{F'}_{A,B}(\relt) = \lift{F}_{A,B}(\relt)$. 
\end{itemize}
It is easy to check that $F'$ preserves relational identities, composition and converse, and, moreover it satisfies 
$F' \circ \GrAr[\RDoc] = F$. 
The only non-trivial commutativity is at the level of functors between bases. 
Take $\fun{f}{A}{B}$ in \CC, its value under $\fn{F'\circ\GrAr[\RDoc]}$ is $\fun{e_{\gr{f}}}{\fn{F}A}{\fn{F}B}$. 
Omitting the indexes in the  natural transformation $\lift{F}$ and recalling that 1-arrows preserve graphs, we deduce 
$\gr{e_{\gr{f}}} = \lift{F}(\gr{f}) = \gr{\fn{F}f}$ and so, by \SRUC, 
we conclude have $e_{\gr{f}}=\fn{F}(f)$, as needed. 
Towards a proof that $F'$ is the unique 1-arrow such that $F'\circ\GrAr[\RDoc] = F$, 
consider $\oneAr{G}{\Ruc{\RDoc}}{\SDoc}$ is such that $G\circ\GrAr[\RDoc] = F$. 
Then, since $G$ preserves graphs, this equation implies that that  the graph of $\fn{G}(\relr)$, for \fun\relr{A}{B} in $\Map\RDoc$, 
is equal to $\lift{F}(\alpha)$; hence, 
by \SRUC, we deduce $\fn{G}(\relr) = e_\relr = \fn{F'}(\relr)$, showing that $\fn{F'} = \fn{G}$. 
The equality $\lift{G} = \lift{F'}$ is straightforward, hence we conclude $G = F'$, as needed. 

To conclude the proof, we just need to show tht $\blank\circ\GrAr[\RDoc]$ is fully faithful. 
Let $\oneAr{F,G}{\Ruc\RDoc}{\SDoc}$ be 1-arrows in $\RDtnRuc$ and 
\twoAr{\theta}{F\circ\GrAr[\RDoc]}{G\circ\GrAr[\RDoc]} a 2-arrow in \RDtn. 
We define a 2-arrow \twoAr{\theta'}{F}{G} in $\RDtnRuc$  as $\theta'_X = \theta_X$. 
In order to show that $\theta'$ is a well-defined 2-arow, 
we only check it is a natural transformation, the other condition being immediate. 
Consider an arrow \fun{\relr}{X}{Y} in $\Map\RDoc$ and observe that 
$\gr{\theta'_Y\circ\fn{F}\relr} 
  = \gr{\fn{F}\relr} \rcomp \gr{\theta_Y} 
  \order \gr{\theta_X} \rcomp \gr{\fn{G}\relr} 
  = \gr{\fn{G}\relr \circ \theta'_X}$, 
which,using \cref{prop:tot-fun-discrete} and \SRUC, implies the equality 
$\theta'_Y\circ\fn{F}\relr = \fn{G}\relr\circ\theta'_X$. 
Finally, we have $\theta'\GrAr[\RDoc] = \theta$ and clearly it is unique with this property, thus proving the thesis. 
\end{proof}

\begin{remark}
Note that working with \SRUC instead of \RUC is crucial in the proof of \cref{thm:doc-to-ruc}. 
Indeed, without \SRUC, we would not be able to construct the 1-arrow $F'$ which factorises $F$ along $\GrAr[\RDoc]$, as it is defined by taking the tracking arrow of a certain functional and total relation. Without \SRUC tracking arrows need not be unique, and even with choice, one could not prove, for instance, that the chosen ones preserve compositions and identities in the base category. 
\end{remark}

\cref{thm:doc-to-ruc} determines a 2-adjuntion 
\[\xymatrix@R=7ex@C=7ex{
\RDtnRuc  \ar@{_(->}@/^8pt/[rr] 
          \ar@{}[rr]|{\top} && 
\RDtn     \ar@/^8pt/[ll]^-{\RtoRuc}
}\]
where $\RtoRuc(\RDoc) = \Ruc\RDoc$, 
$\fn{\RtoRuc(F)}X = \fn{F}X$, $\fn{\RtoRuc(F)}\relr = \lift{F}\relr$ and 
$\lift{\RtoRuc(F)}_{X,Y} = \lift{F}_{X,Y}$, and 
$\RtoRuc(\theta)_X = \gr{\theta_X}$. 
This 2-adjunction determines an idempotent 2-monad on \RDtn whose 2-category of algebras is isomorphic to $\RDtnRuc$, 
thanks to  \cref{prop:sruc}. 
Indeed, every algebra structure on a relational doctrine $\RDoc$, having $\GrAr[\RDoc]$ as section, 
ensures that $\RDoc$ satisfies \SRUC. 
Furhtermore, such algebras are necessarily inverses of $\GrAr[\RDoc]$, hence uniquely determined. 
This shows that satisfying \SRUC is a property rather than a structure. 

\begin{remark} 
Most categorical models for the calculus of relations, such as allegories \cite{FreydS90} or (locally posetal) cartesian bicategories \cite{CarboniW87}, 
are some kind of  \emph{ordered category with involution} \cite{Lambek99}. 
These are \Pos-enriched categories \CC together with a \Pos-enriched dagger, \ie, an identity on objects and self inverse \Pos-functor \fun{(\blank)\rconv}{\CC\op}{\CC}. 
In \cite{DagninoP23} it is shown that the 2-category of 
ordered categories with involution, 
dagger preserving \Pos-enriched functors and lax natural transformations is equivalent to 
the 2-category $\RDtnRuc$ of relational doctrines satisfying \SRUC. 
Therefore, from \cref{thm:doc-to-ruc} we also derive that ordered categories with involution arise as  algebras for the idempotent 2-monad on relational doctrines mentioned above.
% \[\xymatrix@R=7ex@C=7ex{
% \OCI  \ar@/^8pt/[r]  \ar@{}[r]|{\top} & \RDtn \ar@/^8pt/[l]
% }\]
\end{remark}

% !TEX root = main.tex

\section{Completions and Singletons}
\label{sect:singletons}

In the previous section, we have seen two ways of constructing a relational doctrine satisfying \SRUC starting from an arbitrary relational doctrine $\RDoc$: 
by restricting to strongly Cauchy-complete objects ($\complete{\RDoc}$) or by replacing arrows with functional and total relations ($\Ruc\RDoc$). 
A natural question that arises at this point is when these two are equivalent. 
In this section, we will show that $\complete\RDoc$ and $\Ruc\RDoc$ are equivalent in the 2-category \RDtn exactly when 
one of the following equivalent conditions holds:
\begin{itemize}
\item $\RDoc$ has a \emph{Cauchy-completion}, \ie, a universal construction turning any object into a strongly Cauchy-complete one; 
\item $\RDoc$ has \emph{singleton objects}, \ie, objects classifying functional and total relations. 
\end{itemize}
Both of these structures are described in terms of certain adjunctions in \RDtn, 
hence, we start by reviewing some of their properties. 

An adjunction in \RDtn (see \eg, \cite{DagninoR21}) consists of the following data: 
two 1-arrows \oneAr{F}{\RDoc}{\SDoc} and \oneAr{\SDoc}{\RDoc} and 
two 2-arrows \twoAr{\un}{\Id_\RDoc}{GF} and \twoAr{\cun}{FG}{\Id_\SDoc}, 
satisfying the usual identities 
$(\cun F)(F\un) = \id_F$ and $(G\cun)(\un G) = \id_G$. 
In other words, this means that we have 
\begin{itemize}
\item an adjunction \ple{\fn{F},\fn{G},\un,\cun} in \Ct{Cat} and 
\item for every $\relr\in\RDoc(X,Y)$ and $\rels\in\SDoc(A,B)$, 
$\relr \order \RDoc\reidx{\un_X,\un_Y}(\lift{G}_{\fn{F}X,\fn{F}Y}(\lift{F}_{X,Y}(\relr)))$ and 
$\lift{F}_{\fn{G}A,\fn{G}B}(\lift{G}_{A,B}(\rels)) \order \SDoc\reidx{\cun_A,\cun_B}(\rels)$. 
\end{itemize}

A 1-arrow \oneAr{F}{\RDoc}{\SDoc} in \RDtn is a \emph{change-of-base} if 
\nt{\lift{F}}{\RDoc}{\SDoc\bop{\fn{F}}} is a natural isomorphism. 
Intuitively, a change-of-base is a 1-arrow that acts only on base categories, 
leaving posets of relations unchanged. 
For example, the universal 1-arrow \oneAr{\GrAr[\RDoc]}{\RDoc}{\Ruc\RDoc} is a change-of-base. 

Adjunctions that involve a change-of-base 1-arrow have  special properties as the following proposition shows. 

\begin{proposition}\label[prop]{prop:cb-adj}
Let \oneAr{F}{\RDoc}{\SDoc} be a change-of-base and 
\oneAr{G}{\SDoc}{\RDoc} a right adjoint of $F$ in \RDtn. 
Then, the following hold: 
\begin{enumerate}
\item\label{prop:cb-adj:1}
for all objects $A,B$ in the base of $\SDoc$, we have 
$\lift{G}_{A,B} = \lift{F}^{-1}_{\fn{G}A,\fn{G}B} \circ \SDoc\reidx{\cun_A,\cun_B}$
\item\label{prop:cb-adj:2}
$G$ is a change-of-base if and only if each component of $\cun$ is \bijective\SDoc
\end{enumerate}
\end{proposition}
\begin{proof}
\cref{prop:cb-adj:1}. 
Let $A,B$ be objects in the base of $\SDoc$. 
Since \twoAr\cun{FG}{\Id_\SDoc} is a 2-arrow in \RDtn, we have 
$\lift{F}_{\fn{G}A,\fn{G}B} \circ \lift{G}_{A,B} \order \SDoc\reidx{\cun_A,\cun_B}$, which is equivalent to 
$\lift{G}_{A,B} \order \lift{F}^{-1}_{\fn{G}A,\fn{G}B} \circ \SDoc\reidx{\cun_A,\cun_B}$. 
Hence, it suffices to show the opposite inequality. 
We have 
\begin{align*}
\lift{F}^{-1}_{\fn{G}A,\fn{G}B}\circ \SDoc\reidx{\cun_A,\cun_B} 
  &\order \RDoc\reidx{\un_{\fn{G}A},\un_{\fn{G}B}} \circ \lift{G}_{\fn{F}\fn{G}A,\fn{F}\fn{G}B} \circ \lift{F}_{\fn{G}A,\fn{G}B} \circ 
        \lift{F}^{-1}_{\fn{G}A,\fn{G}B} \circ \SDoc\reidx{\cun_A,\cun_B} 
    &&\text{$\un$ is a 2-arrow} 
  \\
  &= \RDoc\reidx{\un_{\fn{G}A},\un_{\fn{G}B}} \circ \lift{G}_{\fn{F}\fn{G}A,\fn{F}\fn{G}B} \circ \SDoc\reidx{\cun_A,\cun_B} \\
  &= \RDoc\reidx{\un_{\fn{G}A},\un_{\fn{G}B}} \circ \RDoc\reidx{\fn{G}\cun_A,\fn{G}\cun_B} \circ \lift{G}_{A,B}  
    &&\text{$\lift{G}$ is natural} 
  \\
  &= \RDoc\reidx{(\fn{G}\cun_A)\un_{\fn{G}A}, (\fn{G}\cun_B)\un_{\fn{G}B}} \circ \lift{G}_{A,B} 
    &&\text{$\RDoc$ is a functor}
  \\ 
  &= \lift{G}_{A,B}  
    &&\text{triangle identities} 
\end{align*}

\cref{prop:cb-adj:2}. 
If $G$ is a change-of-base, then $\lift{G}$ is a natural isomorphism, hence we get 
$\SDoc\reidx{\cun_A,\cun_A} = \lift{F}_{\fn{G}A,\fn{G}A}\circ\lift{G}_{A,A}$. 
This implies that $\SDoc\reidx{\cun_A,\cun_A}$ is an isomorphism and so $\cun_A$ is \bijective\SDoc by \refItem{prop:arrows-iff}{bij}. 
Conversely, if $\cun$ is componentwise \bijective\SDoc, 
by \refItem{prop:arrows}{bij}, $\SDoc\reidx{\cun_A,\cun_B}$ is an isomorphism and so, 
by \cref{prop:cb-adj:1}, we conclude 
$\lift{G}_{A,B} = \lift{F}^{-1}_{\fn{G}A,\fn{G}B} \circ \SDoc\reidx{\cun_A,\cun_B}$ is an isomorphism as well. 
\end{proof}

\cref{prop:cb-adj} essentially states that  the action on the fibres of a right adjoint of a change-of-base 1-arrow is completely determined by  the left adjoint and the counit of the adjunction. 
Moreover, the right adjoint is itself a change-of-base exactly when the counit is bijective. 
Thanks to the duality of relational doctrines described at the end of \cref{sect:rel-doc}, we can prove a similar result for left adjoints of change-of-base 1-arrows. 

\begin{corollary}\label[cor]{cor:cb-adj}
Let \oneAr{F}{\RDoc}{\SDoc} be a change-of-base and 
\oneAr{G}{\SDoc}{\RDoc} a left adjoint of $F$ in \RDtn. 
Then, the following hold: 
\begin{enumerate}
\item\label{cor:cb-adj:1}
for all objects $A,B$ in the base of $\SDoc$, we have 
$\lift{G}_{A,B} = \lift{F}^{-1}_{\fn{G}A,\fn{G}B} \circ \Ex^\SDoc\reidx{\un_A,\un_B}$
\item\label{cor:cb-adj:2}
$G$ is a change-of-base if and only if each component of $\un$ is \bijective\SDoc
\end{enumerate}
\end{corollary}
\begin{proof}
It follows by \cref{prop:cb-adj}, noting that 
\oneAr{G\rop}{\SDoc\rop}{\RDoc\rop} is right adjoint to \oneAr{F\rop}{\RDoc\rop}{\SDoc\rop}, moreover 
\twoAr{\un\rop}{F\rop G\rop}{\Id_{\RDoc\rop}} is the counit and 
$\SDoc\rop\reidx{\un\rop_A,\un\rop_B} = \Ex^\SDoc\reidx{\un_A,\un_B}$. 
\end{proof}

We are now ready to tackle our problem. 
Let \fun{\RDoc}{\bop{\CC}}{\Pos} be a relational doctrine. 
We have the following commutative diagram of 1-arrows in \RDtn
\[\xymatrix{
& \Ruc\RDoc 
\\
  \complete\RDoc
  \ar[ru] 
  \ar@{^(->}[r]_-{\IncAr[\RDoc]}
& \RDoc 
  \ar[u]_-{\GrAr[\RDoc]} 
}\]
where both $\IncAr[\RDoc]$ and $\GrAr[\RDoc]$ are change-of-base and identity-on-objects, hence, 
so is their composition $\GrAr[\RDoc]\circ\IncAr[\RDoc]$. 

A Cauchy-completion of $\RDoc$ is a change-of-base left adjoint of $\IncAr[\RDoc]$. 
Intuitively, this means that we can turn any object of \CC into a strongly Cauchy-complete one in such a way that 
relations between completed objects are the same as those between the original ones. 
Note that, since $\fn{\IncAr[\RDoc]}$ is fully faithful, a Cauchy-completion exhibits $\complete\RDoc$ as a \emph{reflective subdoctrine} of $\RDoc$.

\begin{theorem}\label[thm]{thm:sruc-completion}
For every relational doctrine $\fun{\RDoc}{\bop\CC}{\Pos}$ the following are equivalent:
\begin{enumerate}
\item $\GrAr[\RDoc] \circ \IncAr[\RDoc]$ is an equivalence in \RDtn 
\item $\IncAr[\RDoc]$ has a change-of-base left adjoint in \RDtn 
\end{enumerate}
\end{theorem}
 \begin{proof}
$1\Rightarrow 2$. 
Let \oneAr{F}{\Ruc\RDoc}{\complete\RDoc} be the pseudoinverse of $\GrAr[\RDoc]\circ\IncAr[\RDoc]$. 
By \cref{cor:cb-adj}, it suffices to prove that $\fn{\IncAr[\RDoc]}$ has a left adjoint in \Ct{Cat} and the unit of such adjuntion is componentwise \bijective\RDoc. 
Denote by $\SCompl$ the functor \fun{\fn{F}}{\Map\RDoc}{\complete[\RDoc]{\CC}}. 
Every object $A$ in $\CC$ is also an object of $\Map\RDoc$, so it is mapped by $\SCompl$ to an object $\SCompl(A)$ of $\complete[\RDoc]{\CC}$. 
Since $\fn{\GrAr[\RDoc]}\circ \fn{\IncAr[\RDoc]}$ is the identity on objects and $\SCompl$ is its pseudoinverse, 
there is also an isomorphism \fun{\relt_A}{A}{\SCompl(A)} in $\Map\RDoc$, that is, 
a bijective relation $\relt_A \in \RDoc(A,\SCompl(A))$. 
By hypothesis $\SCompl(A)$, being an object of $\complete[\RDoc]{\CC}$, is strongly Cauchy-complte, 
hence, there is an arrow $\fun{\eta_A}{A}{\SCompl(A)}$ in $\CC$ such that $\gr{\eta_A}=\relt_A$. 
Now, consider an arrow $\fun{f}{A}{Y}$ in $\CC$ where $Y$ is strongly Cauchy-complte, \ie, an object ot in $\complete[\RDoc]{\CC}$. 
The relation $\gr{\eta_A}\rconv\rcomp\gr{f} = \relt_A\rconv\rcomp\gr{f} \in \RDoc(\SCompl(A),Y)$ is functional and total as it is the composition of two functional and total relations. So there is a unique arrow $\fun{\overline{f}}{\SCompl(A)}{Y}$ in \CC such that $\gr{\overline{f}} = \gr{\eta_A}\rconv\rcomp\gr{f}$. 
Then, we have 
$\gr{\overline{f}\circ\eta_A} 
  = \gr{\eta_A}\rcomp\gr{\overline{f}}
  = \gr{\eta_A}\rcomp\gr{\eta_A}\rconv\rcomp\gr{f} 
  = \gr{f}$. 
By \cref{prop:sruc-vs-ruc}, $Y$ is also extensional, hence 
we deduce $\overline{f}\circ \eta_A = f$ in $\CC$. 
Furthermore, 
any other arrow $\fun{g}{\SCompl(A)}{Y}$ in \CC such that $g\circ \eta_A = f$ satisfies 
$\gr{g} 
  = \gr{\eta_A}\rconv\rcomp\gr{\eta_A}\rcomp\gr{g} 
  = \gr{\eta_A}\rconv\rcomp\gr{f} 
  = \gr{\overline{f}}$, 
hence, again by extensionality of $Y$, we conclude $g = \overline{f}$, 
proving that $\overline{f}$ is unique.
This proves that $\SCompl$ is a left adjoint of $\IncAr[\RDoc]$. 
Therefore, the thesis follows as $\eta_A$ is \bijective\RDoc by construction as $\gr{\eta_A} = \relt_A$. 

$2\Rightarrow 1$. 
Let \oneAr{F}{\RDoc}{\complete\RDoc} be the left adjoint of $\IncAr[\RDoc]$ and $\un$ and $\cun$ the unit and counit of the adjunction, respectively. 
By hypothesis, $\un_A$ is \bijective\RDoc for every object $A$ in \CC. 
Since $\fn{\IncAr[\RDoc]}$ is fully faithful, $\cun_A$ is an isomorphism for every object $A$ in $\complete[\RDoc]{\CC}$. 
Applying the 2-functor $\RtoRuc$, obtained after \cref{thm:doc-to-ruc}, 
we obtain an adjunction $\RtoRuc(F)\dashv \RtoRuc(\IncAr[\RDoc])$ between $\Ruc{\complete\RDoc}$ and $\Ruc\RDoc$, where the unit and counit  are given by 
$\RtoRuc(\un)_A = \gr{\un_A}$ and $\RtoRuc(\cun)_Y = \gr{\cun_Y}$ for $A$ in \CC and $Y$ in $\complete[\RDoc]{\CC}$. 
Since $\gr{\un_A}$ is an isomorphism in $\Map\RDoc$, as $\un_A$ is \bijective\RDoc by hypothesis, 
and $\gr{\cun_Y}$ is an isomorphism in $\Map{\complete\RDoc}$, as $\cun_A$ is an isomorphism in $\complete[\RDoc]{\CC}$, 
we get that \oneAr{\RtoRuc(\IncAr[\RDoc])}{\Ruc{\complete\RDoc}}{\Ruc\RDoc} is an equivalence. 
Furthermore, since $\complete\RDoc$ satisfies \SRUC, the 1-arrow \oneAr{\GrAr[\complete\RDoc]}{\complete\RDoc}{\Ruc{\complete\RDoc}} is an isomorphism by \cref{prop:sruc}. 
Hence, the composition $\RtoRuc(\IncAr[\RDoc]) \circ \GrAr[\complete\RDoc] = \IncAr[\RDoc] \circ \GrAr[\RDoc]$ is an equivalence, as needed. 
\end{proof}

\begin{corollary}\label[cor]{cor:sruc-completion}
Let $\CC \hookrightarrow \D$ be a reflective subcategory, with reflector \fun{\SCompl}{\D}{\CC}, and 
\fun{\RDoc}{\bop\CC}{\Pos} a relational doctrine satisfying \SRUC. 
Then, $\RDoc$ and $\complete{\RDoc\bop\SCompl}$ are equivalent. 
\end{corollary}
\begin{proof}
Let $\SDoc = \RDoc\bop\SCompl$. 
By \cref{thm:sruc-completion}, it suffices to show that $\Ruc\RDoc$ and $\Ruc\SDoc$ are equivalent. 
The inclusion $\CC\hookrightarrow\D$ induces a change-of-base \oneAr{F}{\RDoc}{\SDoc}, where $\lift{F}_{X,Y} = \RDoc\reidx{\cun_X,\cun_Y}$, which is an isomorphism as $\cun$ is the counit of a reflection. 
Similarly, the reflector $\SCompl$ induces a change-of-base \oneAr{G}{\SDoc}{\RDoc}, where $\lift{G}_{A,B} = \id_{\RDoc(\SCompl(A),\SCompl(B))}$. 
Then, the reflection extends to an adjunction $G\dashv F$ in \RDtn, whose unit and counit are both bijetive by \cref{prop:cb-adj,cor:cb-adj}. 
Hence, the 2-functor $\RtoRuc$ maps this adjunction to an equivalence between $\Ruc\RDoc$ and $\Ruc\SDoc$ as needed. 
\end{proof}

Let us now focus on singletons. 
In set-theoretic terms, the set of singletons on a set $A$ is a subset $\Sing(A)$ of the powerset $\PP(A)$. 
The powerset $\PP(A)$ is completely characterised by the following universal property: 
it is the set \emph{classifying relations into $A$}, in the sense that for every set $X$ and relation $\relr\subseteq X\times A$ there is a unique function \fun{\chi_\relr}{X}{\PP(A)} such that $\ple{x,a}\in\relr$ if and only if $a\in \chi_\relr(x)$. 
In other words, this means that there is a natural bijection 
$\PP(A)^X \cong \PP(X\times A)$. Since relations are the arrows of the category  $\RelCatSet$ of sets and relations, this gives rise to a natural bijection 
$\Hom\Set{X}{\PP(A)} \cong \Hom{\RelCatSet}{X}{A}$, establishing an adjoint situation between $\Set$ and $\RelCatSet$. 
%
%that suggests a way to define  power-objects in the context of relational doctrines. Let $\fun{\RDoc}{\bop\CC}{\Pos}$ be a relational doctrine and denote by $\RelCat{\CC}$ the category whose objects are those of $\CC$ and where an arrow from $X$ to $A$ is an element of $\RDoc(X,A)$ (composition and identities are the relational ones). The category $\Map{\CC}$ is a subcategory of $\RelCat{\CC}$ and the composition of the functor $\fun{\gr{}}{\CC}{\Map{\CC}}$ with the inclusion $\fun{}{\Map{\CC}}{\RelCat{\CC}}$ gives a functor that, abusing with notation, we denote by $\fun{\gr{}}{\CC}{\RelCat{\CC}}$. It maps $\fun{f}{X}{A}$ to $\fun{\gr{f}}{X}{A}$.
%
% \begin{definition}\label[def]{def:rel-pow}
%A relational doctrine $\fun{\RDoc}{\bop\CC}{\Pos}$ has power-objects if $\fun{\gr{}}{\CC}{\RelCat{\CC}}$ has a right adjoint $\fun{\PPP}{\RelCat{\CC}}{\CC}$. 
% \end{definition}
%
%The functor $\PPP$ maps an object $A$ of $\CC$ to another object $\PPP (A) $ of $\CC$ which will be called power-object of $A$; the counite of the adjunction will be denoted by $\ni$, its value at an object $A$ of $\CC$ is $\ni_A$ in $\RDoc(\PPP A, A)$ and will be called membership predicate over $A$. It has the property that for every $X$ and every $\relr$ in $\RDoc(X,A)$ there is a unique $\fun{\chi_\relr}{X}{\PPP(A)}$ such that $\relr=\RDoc\reidx{\chi_\relr,\id_A}(\ni_A)=\gr{\chi_\relr}\rcomp\ni_A$.
%
% When $\RDoc$ is $\Rel$ and $\CC$ is $\Set$ the power-object of a set $A$ is nothing but its powerset $\PP(A)$. Here, 
% 

Then, for every subset $U$ of $\PP(A)$, the restriction of the natural bijection to $\Hom\Set{X}{U}$ determines a class of relations from $X$ to $A$. 
In particular, when $U$ is the set $\Sing(A)$ of singletons of $A$, the class of relations from $X$ to $A$ it determines is precisely that of functional and total ones. 
That is, there is a natural bijection $\Hom\Set{X}{\Sing(A)} \cong \Hom{{\Map{\Rel}}}{X}{A}$ (where $\Rel$ is the relational doctrine of set-theoretic relations as in \refItem{ex:rel-doc}{vrel}).
In this case, \ie for standard set-theoretic relations,
the adjunction defined by the previous natural bijection is an equivalence, but we cannot require this in general, as it would force the doctrine to satisfy the strong rule of unique choice. 
Hence, we give the following definition. 
 
\begin{definition}\label[def]{def:rel-sing}
A relational doctrine $\fun{\RDoc}{\bop{\CC}}{\Pos}$ has \emph{singleton objects} or simply \emph{singletons} 
if the functor $\fun{\fn{}\GrAr[\RDoc]}{\CC}{\Map{\RDoc}}$ has a fully faithful right adjoint $\fun{\Sing}{\Map{\RDoc}}{\CC}$. 
% We say that $\RDoc$ has singletons if moreover $\Sing$ is full and faithful.
\end{definition}
 
We refer to $\fun{\Sing}{\Map{\RDoc}}{\CC}$ as the \emph{singleton functor}. 
The counit of the adjunction $\fn{\GrAr[\RDoc]}\dashv\Sing$ ensures that, for every object $A$ in $\CC$, there is a 
functional and total relation $\scun_A$ in $\RDoc(\Sing(A),A)$, representing the membership relation between $\Sing(A)$ and $A$, such that, 
for every object $X$ in \CC and every functional and total relation $\relr$ in $\RDoc(X,A)$, 
there is a unique arrow $\fun{\chi_\relr}{X}{\Sing(A)}$ in \CC such that 
$\relr = \gr{\chi_\relr} \rcomp \scun_A$. 
We shall call $\chi_\relr$ the \emph{classifying arrow} of $\relr$ and we will say that $\Sing(A)$ classifies functional and total relations into $A$. 
Since $\Sing$ is fully faithful, we now that $\scun$ is a natural isomorphism, that is, 
each component $\scun_A$ is a bijection. 
This fact captures key properties of singletons as described \eg, in \cite{Fourman1983-FOUSAL}. 
Functionality of $\scun_A$ captures the fact that if two elements belong to the same singleton, then they are equal, 
while totality says that  no singleton is empty. 
Injectivity of $\scun_A$ models the property that 
if two singletons have an element in common, then  they must be equal, 
while surjectivity ensures that for every element of $A$ there is a singleton it belongs to. 

Furthermore, since the 1-arrow \oneAr{\GrAr[\RDoc]}{\RDoc}{\Ruc\RDoc} is a change-of-base, 
by \cref{prop:cb-adj}, we know that the singleton functor extends uniquely to a right adjoint 1-arrow \oneAr{\SingAr}{\Ruc\RDoc}{\RDoc}, which is a change-of-base as $\scun$ is componentwise \bijective{\Ruc\RDoc}; 
in turn, by \cref{cor:cb-adj}, this implies that the unit of the adjunction, denoted by $\sun$, is componentwise \bijective\RDoc, 
because $\GrAr[\RDoc]$ is a change-of-base left adjoint. 

The fact that $\Sing(A)$ classifies functional and total relations into $A$ traces a connection between singletons and Cauchy-complete objects. 
Roughly, if $\Sing(A)$ and $A$ are isomorphic, then classifying arrows provide candidates for tracking arrows. 
This connection is made precise in \cref{thm:singletons}. 
We start by proving that singleton objects are strongly Cauchy-complete. 

\begin{lemma}\label[lem]{lem:sruc-sing} 
Let $\fun{\RDoc}{\bop\CC}{\Pos}$ be a relational doctrine with singleton objects. 
Then, the singleton 1-arrow factors through the inclusion \oneAr{\IncAr[\RDoc]}{\complete{\RDoc}}{\RDoc} 
\[\xymatrix{
& \Ruc\RDoc 
  \ar[d]^-{\SingAr} 
  \ar@{..>}[ld]_-{\SingAr'} 
\\
  \complete{\RDoc}
  \ar[r]_-{\IncAr[\RDoc]}
& \RDoc 
}\]
\end{lemma}
\begin{proof}
It suffices to show that $\Sing(A)$ is strongly Cauchy-complete, as $\fn{\IncAr[\RDoc]}$ is fully faithful and $\lift{\IncAr[\RDoc]}$ is componentwise the identity. 
Consider a functional and total relation $\relr$ in $\RDoc(X,\Sing(A))$. 
By \cref{def:rel-sing}, There is a unique $\fun{\chi_\relr}{X}{\Sing(\Sing(A))}$ with $\relr = \gr{\chi_\relr}\rcomp\scun_{\Sing(A)}$. 
By the triangle identities of the adjunction, we know that 
$\gr{\sun_{\Sing(A)}}\rcomp \scun_{\Sing(A)} = \rid_{\Sing(A)} = \gr{\sun_{\Sing(A)}}\rcomp\gr{\Sing(\scun_A)}$, 
where $\sun$ is the unit of the adjunction $\fn{\GrAr[\RDoc]}\dashv \Sing$. 
Since $\scun$ is a natural isomorphism, $\scun_{\Sing(A)}$ and $\gr{\Sing(\scun_A)}$ are bijections, the latter as it is the graph of an isomorphism. 
Therefore, from the equality above, we deduce that $\scun_{\Sing(A)} = \gr{\Sing(\scun_A)}$ and this implies that 
$\gr{\Sing(\scun_A)\circ \chi_\relr} 
  = \gr{\chi_\relr} \rcomp \gr{\Sing(\scun_A)} 
  = \gr{\chi_\relr}\rcomp \scun_{\Sing(A)} 
  = \relr$, proving that $\fun{\Sing(\scun_A)\circ\chi_\relr}{X}{\Sing(A)}$, is a tracking arrow of $\relr$. 
In order to show that this arrow is unique, 
consider an arrow \fun{f}{X}{\Sing(A)} such that $\relr=\gr{f}$. 
Then from $\gr{\Sing(\scun_A)^{-1}}\rcomp\scun_{\Sing(A)}=\rid_{\Sing(A)}$ it follows that 
$\relr=\gr{f}=\gr{f}\rcomp\rid_{\Sing(A)}=\gr{f}\rcomp\gr{\Sing(\scun_A)^{-1}}\rcomp\scun_{\Sing(A)}$, 
making $\Sing(\scun_A)^{-1}\circ f$ a classifying arrow of $\relr$. 
Therefore, $\Sing(\scun_A)^{-1}f = \chi_\relr$ and so we conclude $f=\Sing(\scun_A)\chi_\relr$.
\end{proof}

Actually, we can prove the following slightly stronger result: 
every strongly Cauchy-complete object is a singleton object. 

\begin{corollary}\label[cor]{cor:sruc-sing}
Let \fun{\RDoc}{\bop\CC}{\Pos} be a relational doctrine. 
An object $X$ in \CC is strongly Cauchy-complete if and only if \fun{\sun_X}{X}{\Sing(A)} is an isomorphism. 
\end{corollary}
\begin{proof}
The left-to-right implication follows from \cref{prop:sruc-iso}, as $\sun_X$ is always \bijective\RDoc. 
The right-to-left implication is immediate by \cref{lem:sruc-sing}. 
\end{proof}

We can now prove our second characterisation, showing that $\Ruc\RDoc$ and $\complete{\RDoc}$ are equivalent exactly when $\RDoc$ has singleton objects.

\begin{theorem}\label[thm]{thm:singletons}
For every relational doctrine $\fun{\RDoc}{\bop\CC}{\Pos}$ the following are equivalent:
\begin{enumerate}
\item $\GrAr[\RDoc] \circ \IncAr[\RDoc]$ is an equivalence in \RDtn 
\item $\GrAr[\RDoc]$ has a change-of-base right adjoint in \RDtn 
\end{enumerate}
\end{theorem} 
\begin{proof}
$1\Rightarrow 2$. 
Let \oneAr{F}{\Ruc\RDoc}{\complete\RDoc} be the pseudoinverse of $\GrAr[\RDoc]\circ\IncAr[\RDoc]$ and 
let \twoAr{\scun}{\GrAr[\RDoc]\circ\IncAr[\RDoc]\circ F}{\Id_{\Ruc\RDoc}} be the natural isomorphism of the equivalence. 
For every $A$ in \CC we set $\Sing(A) = \fn{F}(A)$. 
Consider an arrow \fun{\relr}{X}{A} in $\Map\RDoc$, \ie, a functional and total relation in $\RDoc(X,A)$. 
Since $\scun_A$ is an isomorphism, hence a bijection, $\scun_A\rconv$ is a functional and total relation in $\RDoc(A,\Sing(A))$, 
hence, $\relr\rcomp \scun_A\rconv$ is a functional and total relation in $\RDoc(X,\Sing(A))$. 
By definition, $\Sing(A)$ is an object of $\complete[\RDoc]{\CC}$, thus it is strongly Cauchy-complete and so we know that 
there is a unique arrow \fun{f}{X}{\Sing(A)} in \CC such that $\gr{f} = \relr \rcomp \scun_A\rconv$, which is equivalent to 
$\gr{f}\rcomp\scun_A = \relr$. 
This proves that $\Sing$ determines a right adjoint of $\fn{\GrAr[\RDoc]}$ with counit given by $\scun$. 
Furthermore, since $\scun$ is componentwise an isomorphism, hence a bijection, we also derive that $\Sing$ is fully faithful and 
extends to a change-of-base right adjoint \oneAr{\SingAr}{\Ruc\RDoc}{\RDoc}, by \cref{prop:cb-adj}. 

$2\Rightarrow 1$. 
By \cref{lem:sruc-sing}, the singleton 1-arrow $\SingAr$ factors through $\IncAr[\RDoc]$, 
producing a 1-arrow \oneAr{\SingAr'}{\Ruc\RDoc}{\complete\RDoc}. 
This determines an adjoint situation $(\GrAr[\RDoc]\circ\IncAr[\RDoc]) \dashv\SingAr'$. 
Since $\Sing$ is fully faithful, the counit is a natural isomorphism and 
by \cref{cor:sruc-sing} the unit is an isomorphism as well. 
Therefore, the adjunction is actually an equivalence. 
\end{proof}

Combining \cref{thm:sruc-completion,thm:singletons} we obtain three equivalent conditions summarised in the following diagram 
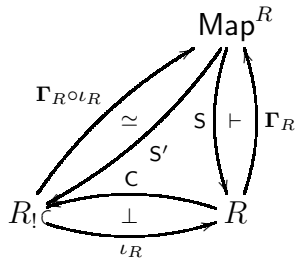
\[\xymatrix{
&& \Ruc\RDoc 
   \ar@/^8pt/[ddll]^-{\SingAr'} 
   \ar@{}[ddll]|{\simeq} 
   \ar@/_8pt/[dd]_-{\SingAr} 
   \ar@{}[dd]|{\vdash} 
\\\\
   \complete\RDoc
   \ar@/^8pt/[rruu]^-{\GrAr[\RDoc]\circ\IncAr[\RDoc]} 
   \ar@{^(->}@/_8pt/[rr]_-{\IncAr[\RDoc]} 
   \ar@{}[rr]|{\bot}
&& \RDoc
   \ar@/_8pt/[uu]_-{\GrAr[\RDoc]}
   \ar@/_8pt/[ll]_-{\SComplAr} 
}\]
where $\SComplAr = \SingAr' \circ \GrAr[\RDoc]$ and 
$\SingAr = \IncAr[\RDoc] \circ \SingAr'$.

\begin{remark}\label[rem]{rem:lousodopo}
The previous diagram gives an easy way to check whether a relational doctrine $\fun{\RDoc}{\bop\CC}{\Pos}$ has singleton or not: it suffices to see if $\complete{\CC}$ is a reflective subcategory of $\CC$ with  \bijective\RDoc unite. Indeed in this case the inclusion of doctrines $\fun{\IncAr[\RDoc]}{\complete{\RDoc}}{\RDoc}$ has a change-of-base left adjoint (\ie $\RDoc$ has a Cauchy-completion) which is build as in \cref{cor:cb-adj} \cref{cor:cb-adj:1}, so $\GrAr[\RDoc] \circ \IncAr[\RDoc]$ is an equivalence by \cref{thm:sruc-completion}. 
\end{remark}

\begin{example} \label[ex]{ex:rel-sing} 
\begin{enumerate}

\item\label{ex:rel-sing:met}
Consider the relational doctrine $\fun{\MetDoc}{\bop{\Met}}{\Pos}$ described in \refItem{ex:rel-doc}{met}. 
As observed in \refItem{ex:extensional}{met} a metric space $Y$ is extensional if and only if it is separated and as observed in \refItem{ex:ruc-doc}{met} the space $Y$ is Cauchy-complete if and only if it is complete. 
Thus the full subcategory of $\Met$ on separated and complete metric spaces is $\complete[\MetDoc]{\Met}$. Note that dense isometries are $\bijective{\MetDoc}$ arrows, indeed an $\fun{f}{A}{B}$ is dense if for all $b\in B$ it is $\inf_{a\in A}\dist_B(f(a),b)=0$ so 
\[
\dist_{B}(b,b')=2\inf_{a\in A}\dist_{B}(f(a),b)+ \dist_{B}(b,b')= \inf_{a\in A}(2\dist_{B}(f(a),b)+ \dist_{B}(b,b'))=
\]
\[
\inf_{a\in A}(\dist_{B}(f(a),b)+\dist_{B}(f(a),b)+ \dist_{B}(b,b'))\geq\inf_{a\in A}(\dist_{B}(f(a),b)+\dist_{B}(f(a),b'))
=\gr{f}\rconv\rcomp\gr{f}(b,b')
\]
showing that $f$ is total. On the other hand if $f$ is an isometry
\[
\gr{f}\rcomp\gr{f}\rconv(a,a')=\inf_{b\in B}(\dist_{B}(f(a),b)+ \dist_{B}(b,f(a)))=\dist_{B}(f(a),f(a'))=\dist_A(a,a')
\]
showing that $f$ is $\injective{\MetDoc}$.
The completion of a metric space provide a left adjoint to the inclusion of $\complete[\MetDoc]{\Met}$ into $\Met$ whose unite at $Y$, \ie the non expansive map $\fun{\eta_Y}{Y}{\overline{Y}}$ is a dense isometry, hence  a family of $\bijective{\MetDoc}$ arrow. So the relational doctrine $\MetDoc$ has singletons by \cref{rem:lousodopo}, then
the category of complete separated metric space equivalent to $\Map{\MetDoc}$.

\item\label{ex:rel-sing:vec}

Consider the relational doctrine $\fun{\SVecDoc}{\bop\ct{SVec}}{\Pos}$ as in \refItem{ex:rel-doc}{vec}.
The category $\Ban$ of Banach spaces is $\complete[\SVecDoc]{\ct{SVec}}$ as shown in \refItem{ex:ruc-doc}{vec}. It is also a reflective subcategory of $\SVecDoc$ where every unit arrow $\fun{\un_X}{X}{\overline{X}}$ is a dense isometry. Dense isometries in $\ct{SVec}$ are $\SVecDoc$-bijections; the proof of this is similar to the one given in \refItem{ex:rel-sing}{met}, as a short linear map $\fun{f}{X}{Y}$ is dense if $\inf_{\vecx\in X}\Norm[Y]{f(\vecx)-\vecy}=0$ and is an isometry if  $\Norm[Y]{f(\vecx)}=\Norm[X]{\vecx}$, which implies that $\Norm[Y]{f(\vecx)-f(\vecx)}=\Norm[Y]{f(\vecx-\vecx')}=\Norm[X]{\vecx-\vecx'}$. So the unite of the adjunction between $\Ban$ and $\ct{SVec}$ is a family of $\SVecDoc$-bijections.
%, indeed 
%\[\gr{f}\rcomp\gr{f}\rconv=\inf_{\vecy\in\Car{Y}} \Norm[Y]{f(\vecx)-\vecy}+\Norm[Y]{f(\vecx')-\vecy} = \Norm[Y]{f(\vecx)-f(\vecx)}=\Norm[Y]{f(\vecx-\vecx')}=\Norm[X]{\vecx-\vecx'}\]
%$$\Norm[Y]{\vecy-\vecy'}=\inf_{\vecx\in\Car{X}} \Norm[Y]{f(x)-y}+\Norm[Y]{f(x)-y'} =\gr{f}\rconv\rcomp\gr{f}$$
%
By \cref{rem:lousodopo} $\SVecDoc$ has singletons and $\Ban$ is equivalent to $\Map{\SVecDoc}$.

\item\label{ex:rel-sing:walters}
Consider the relational doctrine $\fun{\BiMod}{\bop{\BCatss{Rel(\ct{H})}}}{\Pos}$ introduced in \refItem{ex:rel-doc}{walters}, recall that it is extensional (see \refItem{ex:extensional}{walters}) and that $\complete[\BiMod]{(\BCatss{Rel(\ct{H})})}$ is the category of Cauchy-complete symmetric and skeletal $Rel(\ct{H})$-categories (see \refItem{ex:ruc-doc}{walters}). This category is reflexive \cite{Betti1982} and the reflector sends $X$ to $\overline{X}$ where $\Car{\overline{X}}$ is the set of pairs $\ple{h,\alpha}$ where $h\in\ct{H}$ is the one point $Rel(\ct{H})$-category with $e_h(*)=d_h(*,*)=h$ and $\fun{\alpha}{\Car{X}}{\ct{H}}$ is a left adjoint bimodule from $h$ to $X$  (whose right adjoint bimodule is necessarily $\alpha\rconv=\alpha$. Moreover
\[
e_{\overline{X}}(h,\alpha)=h
\qquad
d_{\overline{X}}(\ple{h,\alpha},\ple{g,\beta})=\bigvee_{x\in \Car{X}}\alpha(x)\wedge\beta(x)
\]
Abbreviate by $\eta_{X}^*(x_1)(x_2)$ the function $d_X(x_1,x_2)$, then the unite of the reflection $\fun{\eta_X}{X}{\Car{\overline{X}}}$ sends $x_1$ to $\ple{e_X(x_1),\eta_{X}^*(x_1)}$. Note that 
  $\gr{\eta_X}(x, \ple{h,\alpha})=\alpha(x)$, so $\gr{\eta_X}\rcomp\gr{\eta_X}\rconv\le d_{X}$ and  $d_{\overline{X}}\le \gr{\eta_X}\rconv\rcomp\gr{\eta_X}$.  Thus $\eta$ is a family of $\BiMod$-bijections and by \cref{rem:lousodopo}  the relational doctrine $\BiMod$ has singletons. Therefore the category $\complete[\BiMod]{(\BCatss{Rel(\ct{H})})}$ of Cauchy-complete symmetric and skeletal $Rel(\ct{H})$-categories  is equivalent to $\Map{\BiMod}$.To recover Walters theorem, that says that $\complete[\BiMod]{(\BCatss{Rel(\ct{H})})}$ is the category of sheaves over $\ct{H}$ \cite{W1} note that  for every $X$ in  $\BCatss{Rel(\ct{H})}$ the function $d_X$ is a partial equivalence relation, since symmetry gives $d_X(x_1,x_2)=d_X(x_2,x_1)$ and composition in $Rel(\ct{H})$ gives $d_X(x_1,x_2)\wedge d_X(x_2,x_3)\le d_X(x_1,x_3)$.  Note also that in $Rel(\ct{H})$ it is $\id_u=u$, so
$
e_X(x)=\id_{e_X(x)}\le d_X(x,x)\le e_X(x)\wedge e_X(x)\le e_X(x)
$, showing that $e_X$  is determined by $d_X$. This makes $\BCatss{Rel(\ct{H})}$ a category of partial equivalence relations $X=\ple{\Car{X}, d_X}$ such that $d_X(x_1,x_2)=d_X(x_1,x_1)=d_X(x_2,x_2)$ implies $x_1=x_2$ and where arrows $\fun{f}{X}{Y}$ are functions $\fun{f}{\Car{X}}{\Car{Y}}$ that preserve the partial equivalence relations and such that  $d_X(x,x)=d_Y(f(x),f(x))$. Using terminology as in \cite{PittsCL, OostenJ:reaait}, functions $\fun{\phi}{\Car{X}\times \Car{Y}}{\ct{H}}$ are called predicates, so a $Rel(\ct{H})$-bimodule $\phi$ from $X$ to $Y$ is a relational and strict predicate,
%\ie such that
%\[\phi(x_1,y_1)\wedge d_X(x_1,x_2)\wedge d_Y(y_1,y_2)\le \phi(x_2,y_2)
%\qquad
%\phi(x,y)\le d_X(x,x)\wedge d_Y(y,y)
%\]
 while a left adjoint bimudule  is also a functional and total predicate. Therefore  $\Map{\BiMod}$ is the category of partial equivalence relations  whose arrows from $X$ to $Y$ are the total, functional, relational and strict predicates from $X$ to $Y$. The category $\Map{\BiMod}$ is then the topos of sheaves over $\ct{H}$ as described in \cite{Higgs84} (see also \cite{HylandJ:trit,PittsA:triir,OostenJ:reaait}).

\item\label{ex:rel-sing:topos} For an elementary topos $\ct{E}$ and a topology $j$, the full subcategory $\ct{Shv}_j(\ct{E})$ of $\ct{E}$ on $j$-sheaves is reflective and the unite arrows are $j$-dense.  That is they are $\bijective{\SubDoc_j}$ so $\SubDoc_j$ has singletons and $\ct{Shv}_j(\ct{E})\simeq \complete[\SubDoc_j]{\ct{E}}\simeq\Map{\SubDoc_j}$. 
\item\label{ex:rel-sing:kh}
Consider the relational doctrine $\fun{\TopDoc}{\bop\Top}{\Pos}$ presented in \refItem{ex:rel-doc}{kh}. 
The Stone-Cech compactification provides a left adjoint to the inclusion of the category $\ct{KH}$ of compact-Hausdorff into $\Top$. It is easy to check that unite arrows are \bijective{\TopDoc}, therefore $\TopDoc$ has singletons and $\ct{KH}\simeq\complete[\TopDoc]{\Top}\simeq\Map{\TopDoc}$.
\end{enumerate}
\end{example}

\begin{remark}\label[rem]{rem:frey}
Similar to the situation depicted in \refItem{ex:rel-sing}{walters},  Frey considered in \cite{Jonas} a general construction that has, among its instances, the category whose objects are $\ple{X,\rho}$ where $\fun{\rho}{X\times X}{\ct{H}}$ is a partial equivalence relation and an arrow $\fun{[f]}{\ple{X,\rho}}{\ple{Y,\sigma}}$ is an equivalence class of functions $\fun{f}{X}{Y}$ preserving the partial equivalence relations and where $[f]=[g]$ if $\rho(x_1,x_2)\le\sigma(f(x_1),f(x_2))$. We call this category $\Per{\ct{H}}$. Adding to $\Per{\ct{H}}$ total, functional, relational and strict relations as new morphisms, one finds the   topos  obtained from the localic tripos over $\ct{H}$ by the tripos-to-topos construction \cite{Jonas}, which is known to be the topos of sheaves over $\ct{H}$ \cite{PittsA:triir}. It is proved in \cite{Jonas} that the topos is equivalent to the full subcategory of $\Per{\ct{H}}$ on coarse objects, \ie objects right orthogonal to those arrows that are epic and monic. This result perfectly falls into our setting: the category $\Per{\ct{H}}$ is the base of the relational doctrine that sends $\ple{\ple{X,\rho},\ple{Y,\sigma}}$ to relational and strict relations from $\ple{X,\rho}$ to $\ple{Y,\sigma}$, \ie bimodules in the sense of \refItem{ex:rel-sing}{walters}; this doctrine has singleton objects, and the strongly Cauchy-complete -objects in $\Per{\ct{H}}$ are the coarse ones.
\end{remark}

% !TEX root = main.tex

\section{An application: Stone-Cech compactification via relational doctrines} 
\label{sect:compact} 

In \refItem{ex:rel-doc}{kh} we have considered a relational doctrine over the category of topological spaces, showing then in \refItem{ex:ruc-doc}{kh} and \refItem{ex:rel-sing}{kh} that strongly Cauchy-complete objects in that doctrine are exactly the compact Hausdorff spaces, which admit a Cauchy-completion via the well-known Stone-Cech compactification. 
It is known that the category of topological spaces and continuous maps can be seen as a category of lax algebras for the ultrafilter monad on \Set (with its canonical relational extension) \cite{Barr70}; 
while usual algebras for this monad are precisely compact Hausdorff spaces. 
In monoidal topology \cite{HofmannST14}, this structure is generalised taking an arbitrary monad on \Set and considering $\Qtl$-valued relations, whre $\Qtl$ is a quantale, leading to the notion of \ple{\mnd,\Qtl}-space. 
For such spaces one can identify compact Hausorff ones and describe a Stone-Cech compactification functor \cite{ClementinoH03}. 
In this section, we will rephrase these notions from monoidal topology in the language of relational doctrines, generalising \refItem{ex:rel-doc}{kh}. 
To this end, we will first recall some basic properties of monads on relational doctrines, \ie, monads in the 2-category \RDtn, 
focusing on the associated doctrine of algebras and closed relations, which plays the role of compact Hausdorff spaces. 
Then, we will define a category of lax algebras for the monad, which extends the usual category of algebras, representing all topological spaces. 
Finally, under suitable assumptions, we will construct in elementary terms a Stone-Cech compactification, that is, a left adjoint of the inclusion of usual algebras into lax ones, 
 proving that the former ones are the strongly Cauchy-complete objects of a doctrine on the latter ones. 

Let us fix throughout this section a relational doctrine \fun{\RDoc}{\bop\CC}{\Pos}. 
Following \cite{Street72}, a monad $\mnd = \ple{\mfun,\mun,\mmul}$ on $\RDoc$ consists of 
a 1-arrow \oneAr{\mfun}{\RDoc}{\RDoc} and two 2-arrows 
\twoAr{\mun}{\Id_\RDoc}{\mfun} and \twoAr{\mmul}{\mfun^2}{\mfun}, satisfying the usual monad laws, \ie, 
$\mmul(\mun\mfun) = \id_{\mfun} = \mmul(\mfun\mun)$ and $\mmul(\mmul\mfun) = \mmul(\mfun\mmul)$. 
Adapting from \cite{DagninoR21}, this means that we have 
\begin{itemize}
\item a monad $\fn{\mnd} = \ple{\fn\mfun,\mun,\mmul}$ on the base category \CC such that 
\item for all objects $X,Y$ in \CC and $\relr\in\RDoc(X,Y)$, the following inequalities hold: 
\[
\relr \order \gr{\mun_X}\rcomp\lift\mfun_{X,Y}(\relr)\rcomp\gr{\mun_Y}\rconv 
\qquad 
\lift\mfun_{\fn\mfun X,\fn\mfun Y}(\lift\mfun_{X,Y}(\relr)) \order \gr{\mmul_X}\rcomp\lift\mfun_{X,Y}(\relr)\rcomp\gr{\mmul_Y}\rconv 
\] 
\end{itemize}
We denote by $\EM\CC{\fn\mnd}$ the category of $\fn\mnd$-algebras and their homomorphisms. 
Let \ple{X,a} and \ple{Y,b} be $\fn\mnd$-algebras. 
A relation $\relr$ in $\RDoc(X,Y)$ is \emph{$a,b$-closed} if 
$\gr{a}\rconv\rcomp\lift\mfun_{X,Y}(\relr)\rcomp\gr{g}\order\relr$ or, equivalently, 
$\lift\mfun_{X,Y}(\relr)\rcomp\gr{b}\order\gr{a}\rcomp\relr$. 
We define a functor \fun{\EM\RDoc\mnd}{\bop{\EM\CC{\fn\mnd}}}{\Pos} where 
$\EM\RDoc\mnd(\ple{X,a},\ple{Y,b})$ is the suborder on $a,b$-closed relations and 
$\EM\RDoc\mnd\reidx{f,g} = \RDoc\reidx{f,g}$. 

\begin{proposition}\label[prop]{prop:em-doc}
The functor $\EM\RDoc\mnd$ is a relational doctrine where 
composition, identities and converse are as in $\RDoc$. 
\end{proposition}
\begin{proof}
Let \ple{X,a}, \ple{Y,b} and \ple{Z,c} be $\fn\mnd$-algebras and $\relr\in\EM\RDoc\mnd(\ple{X,a},\ple{Y,b})$ and $\rels\in\EM\RDoc\mnd(\ple{Y,b},\ple{Z,c})$ be closed relations. 
It suffices to check the following facts. 
\begin{itemize}
\item $\relr\rcomp\rels$ is $a,c$-closed. 
We have 
\begin{align*} 
\gr{a}\rconv\rcomp\lift{\mfun}_{X,Z}(\relr\rcomp\rels)\rcomp\gr{c} 
  &= \gr{a}\rconv\rcomp\lift{\mfun}_{X,Y}(\relr)\rcomp\lift{\mfun}_{Y,Z}(\rels)\rcomp\gr{c} 
\\
  &\order \gr{a}\rconv\rcomp\lift\mfun_{X,Y}(\relr)\rcomp\gr{b}\rcomp\gr{b}\rconv\rcomp\lift{\mfun}_{Y,Z}(\rels)\rcomp\gr{c} 
    &&\text{$\gr{b}$ is total} 
\\
  &\order \relr\rcomp\rels
    &&\text{$\relr$ and $\rels$ are closed} 
\end{align*}
\item $\rid_X$ is $a,a$-closed. 
We have
$\gr{a}\rconv\rcomp\lift{\mfun}_{X,X}(\rid_X)\rcomp\gr{a}
  = \gr{a}\rconv\rcomp\gr{a}
  \order \rid_X$, as $\gr{a}$ is functional. 
\item $\relr\rconv$ is $b,a$-closed. 
We have 
$\gr{b}\rconv\rcomp\relr\rconv\rcomp\gr{a} 
  = (\gr{a}\rconv\rcomp\relr\rcomp\gr{b})\rconv
  \order \relr\rconv$, because $\relr$ is $a,b$-closed. 
\end{itemize}
\end{proof}

\begin{remark}
Adapting results in \cite{DagninoR21}, 
we can show that the doctrine $\EM\RDoc\mnd$ is the \emph{Eilenber-Moore object} \cite{Street72} for the monad $\mnd$ in \RDtn, that is, 
for every relational doctrine $\SDoc$, there is a natural isomorphism of categories 
\[ \Hom\RDtn\SDoc{\EM\RDoc\mnd} \cong \Hom{\mathsf{Mnd}(\RDtn)}{\ple{\SDoc,\Id}}{\ple{\RDoc,\mnd}} \]
where $\mathsf{Mnd}(\RDtn)$ is the 2-category of monads in $\RDtn$. 
Note that, in order to prove this result, the fact that  relational doctrines have left adjoints along all arrows is essential. 
\end{remark}

We now define the category $\SP\mnd$ of \emph{$\mnd$-spaces} as follows: 
\begin{description}
\item[objects] are pairs \ple{X,\ralg} where $X$ is an object of \CC and $\ralg$ is relation in $\RDoc(\fn\mfun X,X)$ such that 
\[
\rid_X \order \gr{\mun_X}\rcomp\ralg 
\qquad\text{and}\qquad 
\lift{\mfun}_{\fn\mfun X,X}(\ralg) \rcomp \ralg \order \gr{\mmul_X}\rcomp\ralg 
\]
or, equivalently, 
\[
\gr{\mun_X}\rconv\order\ralg 
\qquad\text{and}\qquad 
\gr{\mmul_X}\rconv\rcomp\lift{\mfun}_{\fn\mfun X,X}(\ralg)\rcomp\ralg \order \ralg 
\]
\item[arrows] \fun{f}{\ple{X,\ralg}}{\ple{Y,\ralgb}} are arrows \fun{f}{X}{Y} in \CC such that 
$\ralg \order \gr{\fn\mfun f} \rcomp \ralgb \rcomp \gr{f}\rconv$ or, equivalently, 
$\ralg \rcomp \gr{f} \order \gr{\fn\mfun f}\rcomp \ralgb$.
\end{description}
Following monoidal topology, the objects \ple{X,\ralg} of this category are regarded as ``convergence spaces'': 
$X$ is the object of points, 
$\fn\mfun X$ is an object of ``generalised sequences'' of points in $X$ and 
$\ralg$ is a \emph{convergence relation}, associating generalised sequences with their limit points. 
Similarly, arrows \fun{f}{\ple{X,\ralg}}{\ple{Y,\ralgb}} are arrows preserving the convergence relation, \ie 
\emph{continuous arrows}. 

We can see $\fn\mnd$-algebras as $\mnd$-spaces  by a functor 
\fun{\SPFun}{\EM\CC{\fn\mnd}}{\SP\mnd}
defined as follows: 
\[
\SPFun\ple{X,a} = \ple{X,\gr{a}}
\qquad 
\SPFun f = f 
\]
Essentially, this functor shows that 
every $\fn\mnd$-algebra can be regarded as a $\mnd$-space where the convergence relation is the graph of an arrow. 
Note also that $\SPFun$ is always faithful, but not necessarily full. 
Indeed, arrows in $\EM\CC{\fn\mnd}$ must satisfy a commutative diagram, while those in $\SP\mnd$ need only to preserve the convergence relation. 
However, the following holds. 

\begin{proposition}\label[prop]{prop:spfun-full}
If $\RDoc$ is extensional, then $\SPFun$ is full. 
\end{proposition}
\begin{proof}
Let \ple{X,a} and \ple{Y,b} be $\fn\mnd$-algebras and 
\fun{f}{\ple{X,\gr{a}}}{\ple{Y,\gr{b}}} be an arrow in \SP\mnd. 
Then, we have 
$\gr{f\circ a} = \gr{a}\rcomp\gr{f} \order \gr{\fn\mfun f}\rcomp\gr{b} = \gr{b\circ\fn\mfun f}$, which, 
by \cref{prop:tot-fun-discrete}, implies $\gr{f\circ a} = \gr{b\circ\fn\mfun f}$. 
Hence, by extensionality we get 
$f\circ a = b\circ \fn\mfun f$, proving that 
\fun{f}{\ple{X,a}}{\ple{Y,b}} is an arrow in $\EM\CC{\fn\mnd}$. 
\end{proof}

A $\mnd$-space \ple{X,\ralg} is said to be \emph{compact Hausdorff} if $\ralg$ is functional and total, that is, 
every generalised sequence  converges exactly to one limit point. 
We denote by \CHSP\mnd the full subcategory of \SP\mnd spanned by compact Hausdorff $\mnd$-spaces. 
Clearly, the functor $\SPFun$ lands into \CHSP\mnd, hence every $\fn\mnd$-algebra, when regarded as a $\mnd$-space, is compact Hausdorff. 
However, the corestriction of $\SPFun$ to \CHSP\mnd, benoted by $\CHSPFun$, is not (essentially) surjective in general, because, without the rule of unique choice,  functional and total relations are not necessarily graphs of arrows. 
Hence, $\fn\mnd$-algebras are compact Hausdorff $\mnd$-spaces with the additional property that convergence can be ``effectively computed'', that is, it is described by an arrow in the base rather than by a relation. 

We can extend the doctrine $\EM\RDoc\mnd$ to \CHSP\mnd in such a way that $\CHSPFun$ extends to a change-of-base 1-arrow in \RDtn. 
We define a functor \fun{\CHEM\RDoc\mnd}{\bop{\CHSP\mnd}}{\Pos}, 
where $\CHEM\RDoc\mnd(\ple{X,\ralg},\ple{Y,\ralgb})$ is the suborder of $\RDoc(X,Y)$ on $\ralg,\ralgb$-closed relations, that is, those $\relr$ in $\RDoc(X,Y)$ such that  $\ralg\rconv\rcomp\lift{\mfun}_{X,Y}(\relr)\rcomp\ralgb \order \relr$, and 
$\CHEM\RDoc\mnd\reidx{f,g} = \RDoc\reidx{f,g}$. 
It is easy to see that $\CHEM\RDoc\mnd$ is a relational doctrine with the relational structure inherited  from $\RDoc$. 

\begin{proposition}\label[prop]{prop:chsp}
\begin{enumerate}
\item\label{prop:chsp:1}
If $\CHEM\RDoc\mnd$ satisfies \SRUC, then $\CHEM\RDoc\mnd$ and $\EM\RDoc\mnd$ are isomorphic. 
\item\label{prop:chsp:2}
If $\RDoc$ satisfies \SRUC, then $\CHEM\RDoc\mnd$ satisfies \SRUC as well. 
\end{enumerate}
\end{proposition}
\begin{proof}
\cref{prop:chsp:1}. 
It suffices to prove that $\CHSPFun$ is an isomorphism of categories. 
Let \ple{X,\ralg} be an object of \CHSP\mnd. 
Then, $\ralg$ is a functional and total relation in $\CHEM\RDoc\mnd(\ple{\fn\mfun X,\gr{\mmul_X}},\ple{X,\ralg})$. 
By \SRUC, we get a unique arrow \fun{a}{\ple{\fn\mfun X,\gr{\mmul_X}}}{\ple{X,\ralg}} such that $\gr{a}=\ralg$. 
By definition of $\mnd$-space and \cref{prop:tot-fun-discrete}, we get the equalities 
$\gr{a\circ\mun_X} = \gr{\id_X}$ and $\gr{a\circ\mmul_X} = \gr{a\circ\fn\mfun a}$. 
Since $\id_X$ and $a\circ \fn\mfun a$ are arrows in \CHSP\mnd, the equalities above imply that 
$a\circ\mun_X$ and $a\circ\mmul_X$ are arrows as well. 
Thus, since \SRUC implies extensionality by \cref{prop:sruc-vs-ruc}, we derive 
$a\circ\mun_X = \id_X$ and $a\circ\mmul_X = a \circ\fn\mfun a$, that is, 
\ple{X,a}  is a $\fn\mnd$-algebra. 
Since $a$ is uniquely determined, this proves that $\CHSPFun$ is bijective on objects. 

Because we already know that $\CHSPFun$ is faithful, to conclude we just need to check that it is full. 
Let \fun{f}{\ple{X,\gr{a}}}{\ple{Y,\gr{b}}} be an arrow in \CHSP\mnd. 
We have that 
$\gr{f\circ a} 
  = \gr{a}\rcomp\gr{f}
  \order \gr{\fn\mfun f} \rcomp \gr{b} 
  = \gr{b\circ\fn\mfun f}$, 
which by \cref{prop:tot-fun-discrete} implies 
$\gr{f\circ a} = \gr{b\circ\fn\mfun f}$. 
Since $f\circ a$ and $b\circ\fn\mfun f$ are parallel arrows in \CHSP\mnd, by extensionality we get $f\circ a = b \circ\fn\mfun f$, proving that 
\fun{f}{\ple{X,a}}{\ple{Y,b}} is a $\fn\mnd$-algebra homomorphism, as needed. 

\cref{prop:chsp:2}. 
Let $\relr$ be a functional and total relation in $\CHEM\RDoc\mnd(\ple{X,\ralg},\ple{Y,\ralgb})$. 
Then, it is a functional and total relation in $\RDoc(X,Y)$ and so by \SRUC there is a unique arrow \fun{f}{X}{Y} in \CC such that 
$\gr{f} = \relr$. 
Now, since $\relr$ is also $\ralg,\ralgb$-closed, we have 
$\ralg\rconv\rcomp\lift\mfun_{X,Y}(\gr{f})\rcomp\gr{\ralgb} \order \gr{f}$, which implies 
$\gr{\fn\mfun f}\rcomp\ralgb \order \ralg\rcomp\gr{f}$ and, by \cref{prop:tot-fun-discrete}, 
$\gr{\fn\mfun f}\rcomp\ralgb = \ralg\rcomp\gr{f}$. 
Therefore, we conclude that 
$\ralg \order \gr{\fn\mfun f}\rcomp\ralgb\rcomp\gr{f}\rconv$, proving that 
\fun{f}{\ple{X,\ralg}}{\ple{Y,\ralgb}} is an arrow in \CHSP\mnd, as needed. 
\end{proof}

In summary, we have two slightly different doctrines that could play the role of compact Hausdorff spaces, 
which coincide when $\RDoc$ has the strong rule of unique choice. 
In the following, we will work with $\EM\RDoc\mnd$, but all results can be easily recasted to $\CHEM\RDoc\mnd$. 
Furthermore, at the end, to apply  \cref{cor:sruc-completion}, we will need to assume that $\RDoc$  satisfies \SRUC, 
thus making the choice between $\EM\RDoc\mnd$ and $\CHEM\RDoc\mnd$ irrelevant. 

We now aim at constructing a left adjoint of $\SPFun$. 
To this end, we will use quotients, so let us start by recalling from \cite{DagninoP23}
how one can talk about quotients within a relational doctrine. 

Let $X$ be an object of \CC. 
An \emph{$\RDoc$-equivalence relation} on $X$ is a relation 
$\eqrelr$ in $\RDoc(X,X)$ satisfying 
$\rid_X\order\eqrelr$ (reflexivity), 
$\eqrelr\rconv\order\eqrelr$ (symmetry) and 
$\eqrelr\rcomp\eqrelr\order\eqrelr$ (transitivity). 
A \emph{quotient arrow} of $\eqrelr$ is an arrow \fun{q}{X}{W} in \CC such that 
$\eqrelr \order \gr{q}\rcomp\gr{q}\rconv$ and, 
for every arrow \fun{f}{X}{Z} with $\eqrelr\order\gr{f}\rcomp\gr{f}\rconv$, there is a unique arrow \fun{h}{W}{Z} such that $f = h\circ q$.  
Roughly, a quotient arrow of $\eqrelr$ is the ``smallest'' arrow $q$ whose kernel $\gr{q}\rcomp\gr{q}\rconv$ is larger than $\eqrelr$, 
that is, mapping $\eqrelr$-equivalent elements to equal ones. 
We say that \emph{$\RDoc$ has quotients} if every $\RDoc$-equivalence relation admits a quotient arrow. 

\begin{remark}
Quotient arrows in general are not very well-behaved and usually one requires them to satisfy also additional conditions such as 
effectiveness and surjectivity \cite{DagninoP23}. 
A quotient arrow $q$ is \emph{effective} if  $\eqrelr = \gr{q}\rcomp\gr{q}\rconv$, 
that is, elements which are equal under $q$ are $\eqrelr$-equivalent, and 
it is \emph{surjective} (a.k.a. descent in \cite{DagninoP23} or of effective descent in \cite{MaiettiME:quofcm}) if $\rid_W = \gr{q}\rconv\rcomp\gr{q}$, 
that is, $q$ is \surjective{\RDoc}. 
We do not need these conditions in the following, hence we do not assume them. 
\end{remark}

A 1-arrow \oneAr{F}{\RDoc}{\SDoc} always preserves equivalence relations: 
if $\eqrelr$ is an $\RDoc$-equivalence relation on $X$, then $\lift{F}_{X,X}(\eqrelr)$ is an $\SDoc$-equivalence relation on $\fn{F}X$. 
Then, when $\RDoc$ and $\SDoc$ have quotients, we say that $F$ preserves them if it maps quotient arrows of $\eqrelr$ in $\RDoc$ to quotient arrows of $\lift{F}_{X,X}(\eqrelr)$ in $\SDoc$. 

From now on, we make the following assumption.

\begin{assumption}\label[asm]{asm:quotients}
$\RDoc$ has quotients and $\fn\mfun$ preserves them. 
\end{assumption}

\begin{remark}
In \cite{DagninoP23} we present a universal construction which freely adds (effective descent) quotients to any relational doctrine. 
Hence, if \cref{asm:quotients} does not hold, we can always force it by applying this quotient completion. 
\end{remark}

The nice fact is that, under \cref{asm:quotients}, we can prove that the doctrine of $\fn\mnd$-algebras and closed relations has quotients as well. 

\begin{theorem}\label[thm]{thm:em-quotients}
$\EM\RDoc\mnd$ has quotients. 
\end{theorem}
\begin{proof}
Let \ple{X,a} be a $\fn\mnd$-algebra and $\eqrelr$ an equivalence relation in $\EM\RDoc\mnd(\ple{X,a},\ple{X,a})$. 
Then, $\eqrelr$ is an $\RDoc$-equivalence relation on $X$, hence, there is a quotient arrow 
\fun{q}{X}{W} in \CC. 
Since $\mfun$ preserves quotients, $\fn\mfun q$ is a quotient arrow of the $\RDoc$-equivalence relation $\lift\mfun_{X,X}(\eqrelr)$ on $\fn\mfun X$. 
Let $f = q\circ a$, then we have 
\begin{align*}
\lift\mfun_{X,X}(\eqrelr) 
  &\order \gr{a}\rcomp\eqrelr\rcomp\gr{a}\rconv 
    &&\text{$\eqrelr$ is $a,a$-closed}
\\
  &\order \gr{a}\rcomp\gr{q}\rcomp\gr{q}\rconv\rcomp\gr{a}\rconv 
    &&\text{$q$ is a quotient}
\\
  &=\gr{f}\rcomp\gr{f}\rconv
\end{align*}
Hence, by the universal property of $q$, there is a unique arrow \fun{b}{\fn\mfun W}{W} making the following diagram commute: 
\[\xymatrix{
  \fn\mfun X 
  \ar[d]_-{a} 
  \ar[r]^-{\fn\mfun q}
& \fn\mfun W 
  \ar@{..>}[d]^-{b}
\\
  X 
  \ar[r]^-{q}
& W 
}\]
If we prove that $b$ is a $\fn\mnd$-algebra, this diagram shows that $q$ is a $\fn\mnd$-algebra homomorphism. 
This follows by the diagrams below, which commute because $a$ is a $\fn\mnd$-algebra and $q$ and $\fn\mfun q$ are quotient arrows. 
\[
\vcenter{\xymatrix{
  X 
  \ar@(ld,lu)[dd]_-{\id_X}
  \ar[d]^-{\mun_X}
  \ar[r]^-{q}
& W 
  \ar@/^8pt/[dd]^-{\id_W}
  \ar[d]_-{\mun_X}
\\
  \fn\mfun X 
  \ar[r]^-{\fn\mfun q}
  \ar[d]^-{a}
& \fn\mfun W
  \ar[d]_-{b}
\\
  X  
  \ar[r]^-{q}
& W
}}\quad 
\vcenter{\xymatrix{
& \fn\mfun^2 X
  \ar[rr]^-{\fn\mfun^2 q}
  \ar[dd]^(.3){\fn\mfun a}
  \ar[ld]_-{\mmul_X}
&
& \fn\mfun^2 W
  \ar[dd]^-{\fn\mfun b}
  \ar[ld]^-{\mmul_W}
\\
  \fn\mfun X 
  \ar[rr]_(.3){\fn\mfun q}
  \ar[dd]_-{a}
&
& \fn\mfun W
  \ar[dd]^(.3){b}
& 
\\
& \fn\mfun X 
  \ar[rr]_(.35){\fn\mfun q}
  \ar[ld]^-{a}
&
& \fn\mfun W
  \ar[ld]^-{b}
\\ 
  X
  \ar[rr]_-{q}
&
& W
}}
\]
To conclude, consider a $\fn\mfun$-algebra homomorphism \fun{g}{\ple{X,a}}{\ple{Z,c}}
such that $\eqrelr \order \gr{g}\rcomp\gr{g}\rconv$. 
Since $q$ is a quotient arrow in $\RDoc$, there is a unique  arrow \fun{h}{W}{Z} such that 
$g = h\circ q$. 
Note that we also have 
$\lift\mfun_{X,X}(\eqrelr) \order \gr{\fn\mfun g}\rcomp\gr{\fn\mfun g}\rconv$, 
as $\mfun$ preserves graphs. 
Therefore, the following diagram, which commutes since $g$ is a $\fn\mnd$-algebra homomorphism and $q$ and $\fn\mfun q$ are quotient arrows in $\RDoc$, 
shows that $h$ is a $\fn\mnd$-algebra homomorphism as well. 
\[\xymatrix{
  \fn\mfun X 
  \ar[rrrd]^-{\fn\mfun q}
  \ar[rd]_-{\fn\mfun g}
  \ar[dd]_-{a}
\\
& \fn\mfun Z
  \ar[rr]_-{\fn\mfun h}
  \ar[dd]^-{c}
&
& \fn\mfun W
  \ar[dd]^-{b}
\\
  X 
  \ar[rrrd]^-{q}
  \ar[rd]_-{g}
\\
& Z 
  \ar[rr]_-{h}
&
& W 
}\]
\end{proof}

We will define the action of the left adjoint of $\SPFun$ on a $\mnd$-space \ple{X,\ralg} as a suitable quotient of the free $\fn\mnd$-algebra \ple{X,\mmul_X} by a suitable equivalence relation, 
which essentially says that two generalised sequences are equivalent if they have a common limit point according to $\ralg$. 
However, in order to construct it, we need a further assumption on $\RDoc$. 

\begin{assumption}\label[asm]{asm:closure}
For every $\fn\mnd$-algebra \ple{X,a} and 
every relation $\relr$ in $\RDoc(X,X)$, there is a reflexive and transitive $a,a$-closed relation $\rtc[a]\relr$ in $\RDoc(X,X)$ such that 
$\relr\order \rtc[a]\relr$ and, 
for every reflexive and transitive $a,a$-closed relation $\rels$ in $\RDoc(X,X)$, if $\relr\order\rels$ then $\rtc[a]\relr \order \rels$. 
\end{assumption}

\begin{remark}
Essentially, \cref{asm:closure} requires that, for every $\fn\mnd$-algebra \ple{X,a}, 
we can compute the best reflexive, transitive and $a,a$-closed overapproximation of every  relation in $\RDoc(X,X)$. 
This can be achieved in many ways. 
For instance, assuming that the fibres of $\RDoc$ are join-semilattices, \cref{asm:closure} is equivalent to requiring that, for every relation $\relr$ in $\RDoc(X,X)$, 
the monotone function $\Phi_{a,\relr}$ on $\RDoc(X,X)$ defined by 
\[
\Phi_{a,\relr}(\relt) = \relr \lor \rid_X \lor \relt\rcomp\relt \lor \gr{a}\rconv\rcomp\lift\mfun_{X,X}(\relt)\rcomp\gr{a} 
\]
has a least (pre-)fixed point. 
For example, this is assured  by one of the following condition. 
\begin{itemize}
\item If the fibres of $\RDoc$ are complete lattices, then the Knaster-Tarski fixed point theorem ensures that every monotone function, like $\Phi_{a,\relr}$ has a least fixed point. 
\item If fibres of $\RDoc$ are suprema of $\omega$-cains, which are preserved by relational composition and by all components of $\lift\mfun$,
then the Kleene fixed point theorem ensures that every $\omega$-continuous function, like $\Phi_{a,\relr}$, has a least fixed point. 
\end{itemize}
\end{remark}

\cref{thm:compactification} provides a left adjoint of the functor \fun{\SPFun}{\EM\CC{\fn\mnd}}{\SP\mnd} 
that plays the role of the Stone-Cech compactification. 

\begin{lemma}\label[lem]{lem:closure-eq}
Let \ple{X,a} be a $\fn\mnd$-algebra. 
If $\relr$ is a symmetric relation in $\RDoc(X,X)$, then $\rtc[a]\relr$ is an $\EM\RDoc\mnd$-equivalence relation on \ple{X,a}. 
\end{lemma}
\begin{proof}
We have that 
$(\rtc[a]\relr)\rconv$ is reflexive as $\rid_X = \rid_X\rconv \order (\rtc[a]\relr)\rconv$, 
$(\rtc[a]\relr)\rconv$ is transitive as $(\rtc[a]\relr)\rconv\rcomp(\rtc[a]\relr)\rconv = (\rtc[a]\relr\rcomp\rtc[a]\relr)\rconv \order (\rtc[a]\relr)\rconv$, 
$\rtc[a]\relr\rconv$ is $a,a$-closed as $\gr{a}\rconv\rcomp\lift\mfun_{X,X}(\rtc[a]\relr\rconv)\rcomp\gr{a} = (\gr{a}\rconv\rcomp\lift\mfun_{X,X}(\rtc[a]\relr)\rcomp\gr{a})\rconv \order \rtc[a]\relr\rconv$ and 
$(\rtc[a]\relr)\rconv$ extends $\relr$ as $\relr = \relr\rconv \order (\rtc[a]\relr)\rconv$ (since $\relr$ is symmetric). 
Therefore, by definition of $\rtc[a]\relr$ we conclude 
$\rtc[a]\relr\order (\rtc[a]\relr)\rconv$, proving that $\rtc[a]\relr$ is symmetric. 
Then, the thesis follows immediately by definition of $\rtc[a]\relr$. 
\end{proof}

\begin{theorem}\label[thm]{thm:compactification}
The functor \fun{\SPFun}{\EM\CC{\fn\mnd}}{\SP\mnd} has a left adjoint
\fun{\SCFun}{\SP\mnd}{\EM\CC{\fn\mnd}}. 
\end{theorem}
\begin{proof}
Let \ple{X,\ralg} be a $\mnd$-space. 
Consider the relation $\eqrelr_\ralg$ in $\RDoc(\fn\mfun X,\fn\mfun X)$ defined by 
\[ \eqrelr_\ralg = \rtc[\mmul_X]{\ralg\rcomp\ralg\rconv} \]
Since $\ralg\rcomp\ralg\rconv$ is symmetric, by \cref{lem:closure-eq}, we have that $\eqrelr_\ralg$ is a $\EM\RDoc\mnd$-equivalence on \ple{\fn\mfun X,\mmul_X}. 
Let \fun{q_\ralg}{\ple{\fn\mfun X,\mmul_X}}{\ple{X_\ralg,a_\ralg}} be a quotient arrow of $\eqrelr_\ralg$, which exists by \cref{thm:em-quotients}, and consider the arrow \fun{\zeta_{\ple{X,\ralg}} = q_\ralg\circ\mun_X}{X}{X_\ralg}. 
We have that 
\begin{align*}
\ralg\rcomp\gr{\zeta_{\ple{X,\ralg}}}
  &= \ralg\rcomp\gr{\mun_X}\rcomp\gr{q_\ralg}  \\
  &\order \ralg\rcomp\ralg\rconv\rcomp\gr{q_\ralg} 
    &&\text{\ple{X,\ralg} is a $\mnd$-space} \\
  &\order \eqrelr_\ralg\rcomp\gr{q_\ralg}  \\
  &\order \gr{q_\ralg} 
    &&\text{$q_\ralg$ is a quotient} \\
  &=\gr{\fn\mfun \mun_X}\rcomp\gr{\mmul_X}\rcomp\gr{q_\ralg} 
    && \mmul_X\circ\fn\mfun\mun_X  = \id_{\fn\mfun X} \\
  &= \gr{\fn\mfun\mun_X}\rcomp\gr{\fn\mfun q_\ralg}\rcomp\gr{a_\ralg} 
    &&\text{$q_\ralg$ is a $\fn\mnd$-algebra homomorphism} \\ 
  &= \gr{\fn\mfun\zeta_{\ple{X,\ralg}}}\rcomp\gr{a_\ralg} 
\end{align*}
proving that \fun{\zeta_{\ple{X,\ralg}}}{\ple{X,\ralg}}{\ple{X_\ralg,\gr{a_\ralg}}} is an arrow in \SP\mnd. 
To conclude, we need to show that $\zeta_{\ple{X,\ralg}}$ is universal. 

Consider a $\fn\mnd$-algebra \ple{Z,c} and an arrow \fun{f}{\ple{X,\ralg}}{\ple{Z,\gr{c}}} in \SP\mnd and set $g = c\circ \fn\mfun f$. 
We have that 
\begin{align*}
\ralg\rcomp\ralg\rconv 
  &\order \ralg\rcomp\gr{f}\rcomp\gr{f}\rconv\rcomp\ralg 
    &&\text{$\gr{f}$ is total} \\ 
  &\order \gr{\fn\mfun f}\rcomp\gr{c}\rcomp\gr{c}\rconv\rcomp\gr{f}\rconv 
    &&\text{$f$ is an arrow in \SP\mnd} \\ 
  &= \gr{g}\rcomp\gr{g}\rconv 
\end{align*}
that is, $\gr{g}\rcomp\gr{g}\rconv$ extends $\ralg\rcomp\ralg\rconv$. 
Moreover,  \fun{g}{\ple{\fn\mfun X,\mmul_X}}{\ple{Z,c}} is a $\fn\mnd$-algebra homomorphism, hence 
$\gr{g}\rcomp\gr{g}\rconv$ is a $\EM\RDoc\mnd$-equivalence relation on \ple{\fn\mfun X,\mmul_X} and, in particular, 
it is reflexive, transitive and $\mmul_X,\mmul_X$-closed. 
Therefore, by definition of $\eqrelr_\ralg$, we deduce that 
$\eqrelr_\ralg \order \gr{g}\rcomp\gr{g}\rconv$. 
By the universal property of $q_\ralg$, we conclude that there is a unique $\fn\mnd$-algebra homomorphism 
\fun{f^\dagger}{\ple{X_\ralg}}{\ple{Z,c}}
making the following diagram commute 
\[\xymatrix{
  X 
  \ar[r]^-{\mun_X}
  \ar[d]_-{f} 
& \fn\mfun X 
  \ar[r]^-{q_\ralg}
  \ar[d]_-{\fn\mfun f}
& X_\ralg
  \ar@{..>}[d]^-{f^\dagger} 
\\
  Z
  \ar[r]^-{\mun_Z}
  \ar@(rd,ld)[rr]_-{\id_Z} 
& \fn\mfun Z
  \ar[r]^-{c}
& Z 
}\]
and this proves the thesis. 
\end{proof}

This results provides an elementary description of a Stone-Cech compactification, requiring few and simple assumptions on the ``logic'' of the doctrine $\RDoc$. 
In particular, we do not require the base category to be complete, while completeness is an essential hypothesis for \eg \cite{ClementinoH03}. 
On the other hand, \cite{ClementinoH03} works with lax monads, that in our context would mean requiring $\mfun$ only to laxly preserve relational composition and identities. 
Therefore, a precise comparison between the two constructions is still an open question, which we left for future work. 

We conclude the paoer, showing that, when $\RDoc$ satisfies the strong rule of unique choice, 
the compact Hausdorff $\mnd$-spaces arise as the strongly Cauchy-complete objects of a relational doctrine on \SP\mnd, obtained by change-of-base along the compactification functor $\SCFun$, thus generalising \refItem{ex:rel-sing}{kh}. 

\begin{corollary}\label[cor]{cor:sc-sruc}
If $\RDoc$ satisfies \SRUC, then 
$\EM\RDoc\mnd$ and $\complete{\EM\RDoc\mnd\bop\SCFun}$ are equivalent. 
\end{corollary}
\begin{proof}
Since $\RDoc$ satisfies \SRUC, by \cref{prop:chsp}, we deduce that $\EM\RDoc\mnd$ satisfies \SRUC as well. 
Moreover, since $\RDoc$ is extensional by \cref{prop:sruc-vs-ruc}, we deduce that $\SPFun$ is full, by \cref{prop:spfun-full}. 
Hence, by \cref{thm:compactification}, $\EM\CC{\fn\mnd}$ is a reflective subcategory of \SP\mnd and so the thesis follows from \cref{cor:sruc-completion}. 
\end{proof}

\begin{example}
Let \fun{\Rel}{\bop\Set}{\Pos} be the doctrine of set-theoretic relations. 
This doctrine has quotients, the fibres are complete lattices and also satisfies the strong rule of unique choice. 
Let $\fn\mnd = \ple{\fn\mfun,\mun,\mmul}$ be a monad on \Set, where $\fn\mfun$ preserves weak pullbacks. 
Then, using the Barr extension \cite{Barr70}, this induces a monad $\mnd$ on $\Rel$ and, 
assuming the Axiom of Choice, this monad preserves quotients in $\Rel$. 
Therefore, \cref{asm:quotients,asm:closure} are satisfied and \cref{cor:sc-sruc} applies. 
\end{example}

% \input{generalisations}

 % Note: Although we accept most reasonable bibliographical styles, the
 % following is the one we most strongly recommend.  It results in the
 % author, year entry in the paper, rather than uninformative numbers in
 % brackets.

\bibliographystyle{plainurl}  %{plainnat}  
\bibliography{biblio} 

\begin{thebibliography}{10}

\bibitem{Givant1}
Hajnal Andréka, Steven Givant, Peter Jipsen, and Istvan Németi.
\newblock On tarski’s axiomatic foundations of the calculus of relations.
\newblock {\em The Journal of Symbolic Logic}, 82(3):966--994, 2017.

\bibitem{Barr70}
Michael Barr.
\newblock Relational algebras.
\newblock In S.~MacLane, H.~Applegate, M.~Barr, B.~Day, E.~Dubuc, Phreilambud,
  A.~Pultr, R.~Street, M.~Tierney, and S.~Swierczkowski, editors, {\em Reports
  of the Midwest Category Seminar IV}, pages 39--55, Berlin, Heidelberg, 1970.
  Springer Berlin Heidelberg.

\bibitem{Barr}
Michael Barr.
\newblock On categories with effective unions.
\newblock In Francis Borceux, editor, {\em Categorical Algebra and its
  Applications}, pages 19--35, Berlin, Heidelberg, 1988. Springer Berlin
  Heidelberg.

\bibitem{Betti1982}
Renato Betti and Aurelio Carboni.
\newblock Cauchy-completion and the associated sheaf.
\newblock {\em Cahiers de Topologie et Géométrie Différentielle
  Catégoriques}, 23(3):243--256, 1982.
\newblock URL: \url{http://eudml.org/doc/91301}.

\bibitem{Borceux1986}
Francis Borceux and Dominique Dejean.
\newblock Cauchy completion in category theory.
\newblock {\em Cahiers de Topologie et Géométrie Différentielle
  Catégoriques}, 27(2):133--146, 1986.
\newblock URL: \url{http://eudml.org/doc/91378}.

\bibitem{pjm/1102700481}
Aurelio Carboni and Ross Street.
\newblock {Order ideals in categories.}
\newblock {\em Pacific Journal of Mathematics}, 124(2):275 -- 288, 1986.

\bibitem{CarboniW87}
Aurelio Carboni and Rober Walters.
\newblock Cartesian bicategories {I}.
\newblock {\em Journal of Pure and Applied Algebra}, 49(1-2):11--32, 1987.

\bibitem{ClementinoH03}
Maria~Manuel Clementino and Dirk Hofmann.
\newblock Topological features of lax algebras.
\newblock {\em Applied Categorical Structures}, 11:267--286, 2003.
\newblock \href {https://doi.org/10.1023/A:1024274315778}
  {\path{doi:10.1023/A:1024274315778}}.

\bibitem{DagninoP22}
Francesco Dagnino and Fabio Pasquali.
\newblock Logical foundations of quantitative equality.
\newblock In Christel Baier and Dana Fisman, editors, {\em Proceedings of the
  37th Annual {ACM/IEEE} Symposium on Logic in Computer Science, {LICS} 2022},
  pages 16:1--16:13. {ACM}, 2022.
\newblock \href {https://doi.org/10.1145/3531130.3533337}
  {\path{doi:10.1145/3531130.3533337}}.

\bibitem{DagninoP23}
Francesco Dagnino and Fabio Pasquali.
\newblock Quotients and extensionality in relational doctrines.
\newblock In Marco Gaboardi and Femke van Raamsdonk, editors, {\em 8th
  International Conference on Formal Structures for Computation and Deduction,
  {FSCD} 2023}, volume 260 of {\em LIPIcs}, pages 25:1--25:23. Schloss Dagstuhl
  - Leibniz-Zentrum f{\"{u}}r Informatik, 2023.
\newblock \href {https://doi.org/10.4230/LIPIcs.FSCD.2023.25}
  {\path{doi:10.4230/LIPIcs.FSCD.2023.25}}.

\bibitem{DagninoR21}
Francesco Dagnino and Giuseppe Rosolini.
\newblock Doctrines, modalities and comonads.
\newblock {\em Mathematical Structures in Computer Science}, 31(7):769--798,
  2021.
\newblock \href {https://doi.org/10.1017/S0960129521000207}
  {\path{doi:10.1017/S0960129521000207}}.

\bibitem{Fourman1983-FOUSAL}
M.~P. Fourman, D.~S. Scott, and C.~J. Mulvey.
\newblock Sheaves and logic.
\newblock {\em Journal of Symbolic Logic}, 48(4):1201--1203, 1983.
\newblock \href {https://doi.org/10.2307/2273683} {\path{doi:10.2307/2273683}}.

\bibitem{Jonas}
Jonas Frey.
\newblock Triposes, q-toposes and toposes.
\newblock {\em Annals of Pure and Applied Logic}, 166(2):232--259, 2015.
\newblock URL:
  \url{https://www.sciencedirect.com/science/article/pii/S0168007214001109},
  \href {https://doi.org/https://doi.org/10.1016/j.apal.2014.10.005}
  {\path{doi:https://doi.org/10.1016/j.apal.2014.10.005}}.

\bibitem{FreydS90}
Peter~J. Freyd and Andre Scedrov.
\newblock {\em Categories, allegories}, volume~39 of {\em North-Holland
  mathematical library}.
\newblock North-Holland, 1990.

\bibitem{Givant06}
Steven Givant.
\newblock The calculus of relations as a foundation for mathematics.
\newblock {\em Journal of Automated Reasoning}, 37(4):277--322, 2006.
\newblock \href {https://doi.org/10.1007/s10817-006-9062-x}
  {\path{doi:10.1007/s10817-006-9062-x}}.

\bibitem{Higgs84}
Denis Higgs.
\newblock Injectivity in the topos of complete heyting algebra valued sets.
\newblock {\em Canadian Journal of Mathematics}, 36(3):550–568, 1984.
\newblock \href {https://doi.org/10.4153/CJM-1984-034-4}
  {\path{doi:10.4153/CJM-1984-034-4}}.

\bibitem{HofmannST14}
Dirk Hofmann, Gavin~J Seal, and Walter Tholen.
\newblock {\em Monoidal Topology: A Categorical Approach to Order, Metric, and
  Topology}, volume 153.
\newblock Cambridge University Press, 2014.

\bibitem{HylandJ:trit}
J.M.E. Hyland, P.T. Johnstone, and A.M. Pitts.
\newblock Tripos {T}heory.
\newblock {\em Math. Proc. Camb. Phil. Soc.}, 88:205--232, 1980.

\bibitem{JacobsB:catltt}
Bart P.~F. Jacobs.
\newblock {\em Categorical Logic and Type Theory}, volume 141 of {\em Studies
  in logic and the foundations of mathematics}.
\newblock North-Holland, 2001.
\newblock URL:
  \url{http://www.elsevierdirect.com/product.jsp?isbn=9780444508539}.

\bibitem{Lambek99}
Joachim Lambek.
\newblock Diagram chasing in ordered categories with involution.
\newblock {\em Journal of Pure and Applied Algebra}, 143(1):293--307, 1999.
\newblock \href {https://doi.org/https://doi.org/10.1016/S0022-4049(98)00115-7}
  {\path{doi:https://doi.org/10.1016/S0022-4049(98)00115-7}}.

\bibitem{Lawvere69}
F.~Wiliam Lawvere.
\newblock Adjointness in foundations.
\newblock {\em Dialectica}, 23:281--296, 1969.
\newblock also available as Repr. Theory Appl. Categ., 16 (2006) 1--16.
\newblock \href {https://doi.org/10.1111/j.1746-8361.1969.tb01194.x}
  {\path{doi:10.1111/j.1746-8361.1969.tb01194.x}}.

\bibitem{Lawvere70}
F.~William Lawvere.
\newblock Equality in hyperdoctrines and comprehension schema as an adjoint
  functor.
\newblock In A.~Heller, editor, {\em Proceedings of the {N}ew {Y}ork
  {S}ymposium on {A}pplication of {C}ategorical {A}lgebra}, pages 1--14.
  American {M}athematical {S}ociety, 1970.

\bibitem{Lawvere73}
{F. William} Lawvere.
\newblock Metric spaces, generalized logic, and closed categories.
\newblock {\em Rendiconti del Seminario Matematico e Fisico di Milano},
  43:135--166, 1973.

\bibitem{MaiettiME:quofcm}
Maria~Emilia Maietti and Giuseppe Rosolini.
\newblock Quotient completion for the foundation of constructive mathematics.
\newblock {\em Logica Universalis}, 7(3):371--402, 2013.
\newblock \href {https://doi.org/10.1007/s11787-013-0080-2}
  {\path{doi:10.1007/s11787-013-0080-2}}.

\bibitem{Maietti-Rosolini16}
Maria~Emilia Maietti and Giuseppe Rosolini.
\newblock Relating quotient completions via categorical logic.
\newblock In Dieter Probst and Peter~Schuster (eds.), editors, {\em Concepts of
  Proof in Mathematics, Philosophy, and Computer Science}, pages 229--250. De
  Gruyter, 2016.

\bibitem{MPR}
M.E. Maietti, F.~Pasquali, and G.~Rosolini.
\newblock Triposes, exact completions, and {H}ilbert's {$\epsilon$}-operator.
\newblock {\em Tbilisi Math. J.}, 10(3):141--166, December 2017.

\bibitem{Pep18}
Fabio Pasquali.
\newblock A characterization of those categories whose internal logic is
  {H}ilbert's $\epsilon$-calculus.
\newblock {\em Annals of Pure and Applied Logic}, 2018.
\newblock URL:
  \url{http://www.sciencedirect.com/science/article/pii/S0168007218301301},
  \href {https://doi.org/https://doi.org/10.1016/j.apal.2018.11.003}
  {\path{doi:https://doi.org/10.1016/j.apal.2018.11.003}}.

\bibitem{Peirce}
Charles~S. Peirce.
\newblock The logic of relatives.
\newblock {\em The Monist}, 7(2):161--217, 1897.

\bibitem{PittsA:triir}
A.M. Pitts.
\newblock Tripos theory in retrospect.
\newblock {\em Mathematical Structures in Computer Science}, 12(3):265--279,
  2002.
\newblock \href {https://doi.org/10.1016/S1571-0661(04)00107-0}
  {\path{doi:10.1016/S1571-0661(04)00107-0}}.

\bibitem{PittsCL}
Andrew~M. Pitts.
\newblock Categorical logic.
\newblock In {\em Handbook of logic in computer science, {V}ol. 5}, volume~5 of
  {\em Handbook of Logic in Computer Science}, pages 39--128. Oxford Univ.
  Press, New York, 2000.

\bibitem{Rosolini2000ANO}
Giuseppe Rosolini.
\newblock A note on cauchy completeness for preorders.
\newblock 2000.
\newblock URL: \url{https://api.semanticscholar.org/CorpusID:11064689}.

\bibitem{MaclaneS:sheigl}
I.~Moerdijk S.~Mac~Lane.
\newblock {\em Sheaves in Geometry and Logic}.
\newblock Springer, 1992.

\bibitem{Shulman08}
Michael Shulman.
\newblock Framed bicategories and monoidal fibrations.
\newblock {\em Theory Appl. Categ.}, 20(18):650–738, 2008.

\bibitem{Street72}
R.~Street.
\newblock The formal theory of monads.
\newblock {\em Journal of Pure and Applied Algebra}, 2(2):149 -- 168, 1972.
\newblock \href {https://doi.org/10.1016/0022-4049(72)90019-9}
  {\path{doi:10.1016/0022-4049(72)90019-9}}.

\bibitem{Street81}
Ross Street.
\newblock Cauchy characterization of enriched categories.
\newblock {\em Rendiconti del Seminario Matematico e Fisico di Milano},
  51(1):217--233, December 1981.
\newblock \href {https://doi.org/10.1007/BF02924823}
  {\path{doi:10.1007/BF02924823}}.

\bibitem{tarski1988formalization}
A.~Tarski and S.R. Givant.
\newblock {\em A Formalization of Set Theory Without Variables}.
\newblock Number v. 41 in A formalization of set theory without variables.
  American Mathematical Soc., 1988.

\bibitem{Tarski41}
Alfred Tarski.
\newblock On the calculus of relations.
\newblock {\em Journal of Symbolic Logic}, 6(3):73--89, 1941.
\newblock \href {https://doi.org/10.2307/2268577} {\path{doi:10.2307/2268577}}.

\bibitem{OostenJ:reaait}
Jaap van Oosten.
\newblock {\em Realizability: {An Introduction to its Categorical S}ide},
  volume 152 of {\em Studies in Logic and the Foundations of Mathematics}.
\newblock North Holland Publishing Company, 2008.

\bibitem{W1}
R.~F.~C. Walters.
\newblock Sheaves and {C}auchy-complete categories.
\newblock {\em Cahiers Topologie G\'eom. Diff\'erentielle}, 22(3):283--286,
  1981.
\newblock Third Colloquium on Categories, Part IV (Amiens, 1980).

\end{thebibliography}

% \refs
% 
% \bibitem [Lamport, 1986]{LUG} L. Lamport, Latex User's Guide \&
% Reference Manual. Addison-Wesley (fifth edition), 1986.
% 
% \endrefs

\end{document}